\documentclass[12pt]{amsart}
\usepackage[T1]{fontenc}
\usepackage{geometry}
\usepackage[latin9]{inputenc}
\usepackage{amsthm}
\usepackage{amsmath}
\usepackage{amssymb}
\usepackage[margin=2cm, labelfont = bf]{caption}
\usepackage{setspace}
\usepackage{esint}
\usepackage{enumerate}
\usepackage{etoolbox}
\usepackage{url}
\usepackage{graphicx, epstopdf}
\usepackage[english]{babel}
\usepackage{xcolor}
\usepackage{soul}
\usepackage{tikz}
\usetikzlibrary{automata,arrows,positioning,calc}

\numberwithin{equation}{section}

\theoremstyle{plain}
\newtheorem{thm}{\protect\theoremname}[section]
\ifx\proof\undefined
\newenvironment{proof}[1][\protect\proofname]{\par
\normalfont\topsep6\p@\@plus6\p@\relax
\trivlist
\itemindent\parindent
\item[\hskip\labelsep
\scshape
#1]\ignorespaces
}{%
\endtrivlist\@endpefalse
}
\providecommand{\proofname}{Proof}
\fi
\theoremstyle{plain}
\newtheorem{lem}[thm]{\protect\lemmaname}
\theoremstyle{plain}
\newtheorem{prop}[thm]{\protect\propositionname}
\theoremstyle{plain}
\newtheorem{conjecture}[thm]{\protect\conjecturename}
\theoremstyle{definition}
\newtheorem{defn}[thm]{\protect\definitionname}
\theoremstyle{remark}
\newtheorem{rem}[thm]{\protect\remarkname}
\theoremstyle{plain}

\theoremstyle{definition}

\numberwithin{figure}{section}

\makeatother

\usepackage{babel}
\providecommand{\conjecturename}{Conjecture}
\providecommand{\corollaryname}{Corollary}
\providecommand{\definitionname}{Definition}
\providecommand{\examplename}{Example}
\providecommand{\lemmaname}{Lemma}
\providecommand{\propositionname}{Proposition}
\providecommand{\remarkname}{Remark}
\providecommand{\theoremname}{Theorem}

\newcommand{\R}{\ensuremath{\mathbb{R}}}

\newcommand{\Z}{\ensuremath{\mathbb{Z}}}

\renewcommand{\P}{\ensuremath{\mathbb{P}}}
\newcommand{\E}{\ensuremath{\mathbb{E}}}

\newcommand{\e}{\ensuremath{\varepsilon}}

\begin{document}

%%%%%%%%%%%%%%%%%%%%%%%%%%%%%%%%%%%%%%%%%%%%%%%%%%%%%%%%%%%%%%%%
%%%%%%%%%%%%%%%%%%%%%%%%%%%%%%%%%%%%%%%%%%%%%%%%%%%%%%%%%%%%%%%%
%					TITLE PAGE
%%%%%%%%%%%%%%%%%%%%%%%%%%%%%%%%%%%%%%%%%%%%%%%%%%%%%%%%%%%%%%%%
%%%%%%%%%%%%%%%%%%%%%%%%%%%%%%%%%%%%%%%%%%%%%%%%%%%%%%%%%%%%%%%%

\title[Lattice models with stored momentum]{Diffusion and superdiffusion in lattice models for colliding particles with stored momentum }

\author{Edward Crane, Sean Ledger and B\'alint T\'oth}

\address{School of Mathematics, University of Bristol}

\thanks{The research of E. C. and S. L. was supported by the Heilbronn Institute for Mathematical Research. The research of B. T. was supported by EPSRC fellowship grant EP/P003656/1.} 


\begin{abstract}
We introduce two discrete models of a collection of colliding particles with stored momentum and study the asymptotic growth of the mean-square displacement of an active particle. We prove that the models are superdiffusive in one dimension (with power law correction) and diffusive in three and higher dimensions. In two dimensions we demonstrate superdiffusivity (with logarithmic correction) for certain anisotropic initial conditions. 
\end{abstract}

\maketitle

%%%%%%%%%%%%%%%%%%%%%%%%%%%%%%%%%%%%%%%%%%%%%%%%%%%%%
%%%%%%%%%%%%%%%%%%%%%%%%%%%%%%%%%%%%%%%%%%%%%%%%%%%%%
%%%%%%%%%%%%     1. INTRODUCTION       %%%%%%%%%%%%%%
%%%%%%%%%%%%%%%%%%%%%%%%%%%%%%%%%%%%%%%%%%%%%%%%%%%%%
%%%%%%%%%%%%%%%%%%%%%%%%%%%%%%%%%%%%%%%%%%%%%%%%%%%%%
\section{Introduction} 
\label{Sect_Intro}

%% Introduction, notation, main results, general chat
%%%%%%%%%%%%%%%%%%%%%%%%%%%%%%%%%%%%%%%%%%%%%%%%%%%%%%%%%%%%%%%%
%%%%%%%%%%%%%%%%%%%%%%%%%%%%%%%%%%%%%%%%%%%%%%%%%%%%%%%%%%%%%%%%
%%%%%%%%%%%%%%%%     INTRO, NOTATION & RESULTS       %%%%%%%%%%%
%%%%%%%%%%%%%%%%%%%%%%%%%%%%%%%%%%%%%%%%%%%%%%%%%%%%%%%%%%%%%%%%
%%%%%%%%%%%%%%%%%%%%%%%%%%%%%%%%%%%%%%%%%%%%%%%%%%%%%%%%%%%%%%%%

We study the asymptotics of the mean-square displacement of two continuous-time random walk models on $\Z^d$. In both models, each vertex of $\mathbb{Z}^d$ is assigned a sleeping particle carrying a momentum vector equal to one of the $2d$ canonical unit vectors. An active particle is placed at the origin and is also assigned a momentum vector. In both models jumps occur in continuous time at rate one and at jump times the active particle first moves in the direction of its momentum vector. In the first model (M1), upon reaching the neighbouring site the active particle falls asleep and wakes up the sleeping particle at that site, and the motion is then repeated for the new active particle (see Figure \ref{Intro_Fig_M1_evolution}). In the second model (M2), upon reaching the neighbouring site the active particle remains awake with probability $\tfrac{1}{2}$, or else it falls asleep and the particle at the site awakens. Alternatively one can consider the process as a random walker carrying an arrow. At jump times the walker moves according to the direction of the arrow in its hand and after taking a step the walker either swaps the hand and site arrows with probability $1$ (M1 model) or probability $\tfrac{1}{2}$ (M2 model). The mean-square displacement of the active particle will be denoted
\[
E(t) := \mathbb{E}[|X_t|^2],
\]
where $X_t$ is the position of the walker at time $t$. This expectation is taken over the random i.i.d.~initial configuration of the momentum vectors.

%% FIG: Realisation of M1 model
%\input{fig_M1_realisation}
\begin{figure}
\begin{center}
\begin{tikzpicture}[scale=0.65]
\draw[fill] (-2,-2) circle [radius = 0.1];
\draw[->, line width=0.4mm] (-2,-2) -- (-2,-2+0.4);
\draw[fill] (-2,-1) circle [radius = 0.1];
\draw[->, line width=0.4mm] (-2,-1) -- (-2,-1-0.4);
\draw[fill] (-2,-0) circle [radius = 0.1];
\draw[->, line width=0.4mm] (-2,-0) -- (-2,0-0.4);
\draw[fill] (-2,+1) circle [radius = 0.1];
\draw[->, line width=0.4mm] (-2,+1) -- (-2+0.4,+1);
\draw[fill] (-2,+2) circle [radius = 0.1];
\draw[->, line width=0.4mm] (-2,+2) -- (-2-0.4,+2);

\draw[fill] (-1,-2) circle [radius = 0.1];
\draw[->, line width=0.4mm] (-1,-2) -- (-1+0.4,-2);
\draw[fill] (-1,-1) circle [radius = 0.1];
\draw[->, line width=0.4mm] (-1,-1) -- (-1,-1-0.4);
\draw[fill] (-1,-0) circle [radius = 0.1];
\draw[->, line width=0.4mm] (-1,-0) -- (-1+0.4,0);
\draw[fill] (-1,+1) circle [radius = 0.1];
\draw[->, line width=0.4mm] (-1,+1) -- (-1,+1+0.4);
\draw[fill] (-1,+2) circle [radius = 0.1];
\draw[->, line width=0.4mm] (-1,+2) -- (-1,+2-0.4);

\draw[fill] (0,-2) circle [radius = 0.1];
\draw[->, line width=0.4mm] (0,-2) -- (0-0.4,-2);
\draw[fill] (0,-1) circle [radius = 0.1];
\draw[->, line width=0.4mm] (0,-1) -- (0-0.4,-1);
\draw[fill] (0,-0) circle [radius = 0.1];
\draw[->, line width=0.4mm] (0,-0) -- (0,0+0.4);
\draw[fill] (0,+1) circle [radius = 0.1];
\draw[->, line width=0.4mm] (0,+1) -- (0,+1-0.4);
\draw[fill] (0,+2) circle [radius = 0.1];
\draw[->, line width=0.4mm] (0,+2) -- (0+0.4,+2);

\draw[fill] (1,-2) circle [radius = 0.1];
\draw[->, line width=0.4mm] (1,-2) -- (1+0.4,-2);
\draw[fill] (1,-1) circle [radius = 0.1];
\draw[->, line width=0.4mm] (1,-1) -- (1,-1-0.4);
\draw[fill] (1,-0) circle [radius = 0.1];
\draw[->, line width=0.4mm] (1,-0) -- (1,0+0.4);
\draw[fill] (1,+1) circle [radius = 0.1];
\draw[->, line width=0.4mm] (1,+1) -- (1-0.4,+1);
\draw[fill] (1,+2) circle [radius = 0.1];
\draw[->, line width=0.4mm] (1,+2) -- (1,+2+0.4);

\draw[fill] (2,-2) circle [radius = 0.1];
\draw[->, line width=0.4mm] (2,-2) -- (2,-2+0.4);
\draw[fill] (2,-1) circle [radius = 0.1];
\draw[->, line width=0.4mm] (2,-1) -- (2+0.4,-1);
\draw[fill] (2,-0) circle [radius = 0.1];
\draw[->, line width=0.4mm] (2,-0) -- (2+0.4,0);
\draw[fill] (2,+1) circle [radius = 0.1];
\draw[->, line width=0.4mm] (2,+1) -- (2,+1-0.4);
\draw[fill] (2,+2) circle [radius = 0.1];
\draw[->, line width=0.4mm] (2,+2) -- (2,+2-0.4);

\draw[fill, color = red] (0.1,0) circle [ radius = 0.1];
\draw[->, line width=0.4mm, color = red] (0.1,0) -- (0.1+0.4,0);
\end{tikzpicture} $\qquad$
\begin{tikzpicture}[scale=0.65]
\draw[fill] (-2,-2) circle [radius = 0.1];
\draw[->, line width=0.4mm] (-2,-2) -- (-2,-2+0.4);
\draw[fill] (-2,-1) circle [radius = 0.1];
\draw[->, line width=0.4mm] (-2,-1) -- (-2,-1-0.4);
\draw[fill] (-2,-0) circle [radius = 0.1];
\draw[->, line width=0.4mm] (-2,-0) -- (-2,0-0.4);
\draw[fill] (-2,+1) circle [radius = 0.1];
\draw[->, line width=0.4mm] (-2,+1) -- (-2+0.4,+1);
\draw[fill] (-2,+2) circle [radius = 0.1];
\draw[->, line width=0.4mm] (-2,+2) -- (-2-0.4,+2);

\draw[fill] (-1,-2) circle [radius = 0.1];
\draw[->, line width=0.4mm] (-1,-2) -- (-1+0.4,-2);
\draw[fill] (-1,-1) circle [radius = 0.1];
\draw[->, line width=0.4mm] (-1,-1) -- (-1,-1-0.4);
\draw[fill] (-1,-0) circle [radius = 0.1];
\draw[->, line width=0.4mm] (-1,-0) -- (-1+0.4,0);
\draw[fill] (-1,+1) circle [radius = 0.1];
\draw[->, line width=0.4mm] (-1,+1) -- (-1,+1+0.4);
\draw[fill] (-1,+2) circle [radius = 0.1];
\draw[->, line width=0.4mm] (-1,+2) -- (-1,+2-0.4);

\draw[fill] (0,-2) circle [radius = 0.1];
\draw[->, line width=0.4mm] (0,-2) -- (0-0.4,-2);
\draw[fill] (0,-1) circle [radius = 0.1];
\draw[->, line width=0.4mm] (0,-1) -- (0-0.4,-1);
\draw[fill] (0,-0) circle [radius = 0.1];
\draw[->, line width=0.4mm] (0,-0) -- (0,0+0.4);
\draw[fill] (0,+1) circle [radius = 0.1];
\draw[->, line width=0.4mm] (0,+1) -- (0,+1-0.4);
\draw[fill] (0,+2) circle [radius = 0.1];
\draw[->, line width=0.4mm] (0,+2) -- (0+0.4,+2);

\draw[fill] (1,-2) circle [radius = 0.1];
\draw[->, line width=0.4mm] (1,-2) -- (1+0.4,-2);
\draw[fill] (1,-1) circle [radius = 0.1];
\draw[->, line width=0.4mm] (1,-1) -- (1,-1-0.4);
\draw[fill] (1,-0) circle [radius = 0.1];
\draw[->, line width=0.4mm] (1,0) -- (1+0.4,0);
\draw[fill] (1,+1) circle [radius = 0.1];
\draw[->, line width=0.4mm] (1,+1) -- (1-0.4,+1);
\draw[fill] (1,+2) circle [radius = 0.1];
\draw[->, line width=0.4mm] (1,+2) -- (1,+2+0.4);

\draw[fill] (2,-2) circle [radius = 0.1];
\draw[->, line width=0.4mm] (2,-2) -- (2,-2+0.4);
\draw[fill] (2,-1) circle [radius = 0.1];
\draw[->, line width=0.4mm] (2,-1) -- (2+0.4,-1);
\draw[fill] (2,-0) circle [radius = 0.1];
\draw[->, line width=0.4mm] (2,-0) -- (2+0.4,0);
\draw[fill] (2,+1) circle [radius = 0.1];
\draw[->, line width=0.4mm] (2,+1) -- (2,+1-0.4);
\draw[fill] (2,+2) circle [radius = 0.1];
\draw[->, line width=0.4mm] (2,+2) -- (2,+2-0.4);

\draw[fill, color = red] (0.9,0) circle [ radius = 0.1];
\draw[->, line width=0.4mm, color = red] (0.9,0) -- (0.9,+0.4);
\end{tikzpicture} $\qquad$
\begin{tikzpicture}[scale=0.65]
\draw[fill] (-2,-2) circle [radius = 0.1];
\draw[->, line width=0.4mm] (-2,-2) -- (-2,-2+0.4);
\draw[fill] (-2,-1) circle [radius = 0.1];
\draw[->, line width=0.4mm] (-2,-1) -- (-2,-1-0.4);
\draw[fill] (-2,-0) circle [radius = 0.1];
\draw[->, line width=0.4mm] (-2,-0) -- (-2,0-0.4);
\draw[fill] (-2,+1) circle [radius = 0.1];
\draw[->, line width=0.4mm] (-2,+1) -- (-2+0.4,+1);
\draw[fill] (-2,+2) circle [radius = 0.1];
\draw[->, line width=0.4mm] (-2,+2) -- (-2-0.4,+2);

\draw[fill] (-1,-2) circle [radius = 0.1];
\draw[->, line width=0.4mm] (-1,-2) -- (-1+0.4,-2);
\draw[fill] (-1,-1) circle [radius = 0.1];
\draw[->, line width=0.4mm] (-1,-1) -- (-1,-1-0.4);
\draw[fill] (-1,-0) circle [radius = 0.1];
\draw[->, line width=0.4mm] (-1,-0) -- (-1+0.4,0);
\draw[fill] (-1,+1) circle [radius = 0.1];
\draw[->, line width=0.4mm] (-1,+1) -- (-1,+1+0.4);
\draw[fill] (-1,+2) circle [radius = 0.1];
\draw[->, line width=0.4mm] (-1,+2) -- (-1,+2-0.4);

\draw[fill] (0,-2) circle [radius = 0.1];
\draw[->, line width=0.4mm] (0,-2) -- (0-0.4,-2);
\draw[fill] (0,-1) circle [radius = 0.1];
\draw[->, line width=0.4mm] (0,-1) -- (0-0.4,-1);
\draw[fill] (0,-0) circle [radius = 0.1];
\draw[->, line width=0.4mm] (0,-0) -- (0,0+0.4);
\draw[fill] (0,+1) circle [radius = 0.1];
\draw[->, line width=0.4mm] (0,+1) -- (0,+1-0.4);
\draw[fill] (0,+2) circle [radius = 0.1];
\draw[->, line width=0.4mm] (0,+2) -- (0+0.4,+2);

\draw[fill] (1,-2) circle [radius = 0.1];
\draw[->, line width=0.4mm] (1,-2) -- (1+0.4,-2);
\draw[fill] (1,-1) circle [radius = 0.1];
\draw[->, line width=0.4mm] (1,-1) -- (1,-1-0.4);
\draw[fill] (1,-0) circle [radius = 0.1];
\draw[->, line width=0.4mm] (1,0) -- (1+0.4,0);
\draw[fill] (1,+1) circle [radius = 0.1];
\draw[->, line width=0.4mm] (1,+1) -- (1,+1+0.4);
\draw[fill] (1,+2) circle [radius = 0.1];
\draw[->, line width=0.4mm] (1,+2) -- (1,+2+0.4);

\draw[fill] (2,-2) circle [radius = 0.1];
\draw[->, line width=0.4mm] (2,-2) -- (2,-2+0.4);
\draw[fill] (2,-1) circle [radius = 0.1];
\draw[->, line width=0.4mm] (2,-1) -- (2+0.4,-1);
\draw[fill] (2,-0) circle [radius = 0.1];
\draw[->, line width=0.4mm] (2,-0) -- (2+0.4,0);
\draw[fill] (2,+1) circle [radius = 0.1];
\draw[->, line width=0.4mm] (2,+1) -- (2,+1-0.4);
\draw[fill] (2,+2) circle [radius = 0.1];
\draw[->, line width=0.4mm] (2,+2) -- (2,+2-0.4);

\draw[fill, color = red] (0.9,1) circle [ radius = 0.1];
\draw[->, line width=0.4mm, color = red] (0.9,1) -- (0.9-0.4,1);
\end{tikzpicture}
\caption{\label{Intro_Fig_M1_evolution} Example of an initialisation (left) and states after one (middle) and two (right) jump times in the M1 model. The active particle is in red.  }
\end{center}
\end{figure}
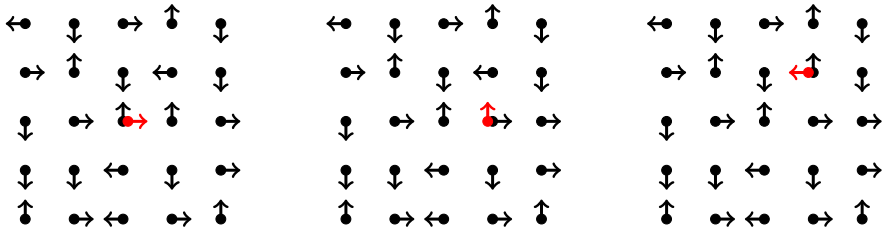

%% End of fig_M1_realisation

For $d=1$, both models are easy to understand. In the M1 model the position of the active particle is ballistic regardless of the initial condition (Theorem \ref{Intro_Thm_D1_NoEps}) and this can be proved through a simple case-by-case analysis. For i.i.d.~initial momentum configurations, the M2 model is precisely the true self-avoiding walk \cite{toth_1995} for which exact $t^{4/3}$-scaling for the mean-square displacement is known (Theorem \ref{Intro_Thm_D1_NoEps}).

A motivation for these models is that they can be seen as natural generalisations of the true-self avoiding walk to higher dimensions. Beyond one dimension, however, we are unable to analyse the M1 and M2 models exactly, and so we introduce an elliptic term into the model generators (Definition \ref{intro_Def_Generators}) to make them tractable. In the resulting versions of the models, which we name $\textrm{M1}_\varepsilon$ and $\textrm{M2}_\varepsilon$, at a small rate the walker will choose to ignore the arrow configuration and take a step uniformly at random. Consequently the modified models are more diffusive. For $d \geq 3$ we are able to show that both new models have mean-square displacements that scale linearly with time (Theorem \ref{Intro_Thm_D3}). The case $d = 2$ is more challenging: we are able to prove upper bounds on the diffusivity of order $t \log t$ for all i.i.d.~ drift-free initial conditions and a superdiffusive lower bound of order $t \sqrt{ \log t}$ for certain anisotropic initial momentum configurations (Theorem \ref{Intro_Thm_D2}).

The mechanism driving the superdiffusion in the models is the persistence of correlations in the configuration  of the arrows. Upon returning to a region the walker has visited previously, the collection of site arrows will be positively correlated with their state on the previous visit. Therefore when the walker returns to a region, it tends to see a similar bias in the local configuration of arrows, and therefore receives a similar drift, when compared to the last visit. This effect is sufficient to cause superdiffusion in one and two dimensions. In three and higher dimensions the underlying dynamics of the walk are transient and so the correlations in the environment of arrows do not influence the asymptotic growth of the mean-squared displacement.

This mechanism is also present in the continuous Brownian polymer \cite{horvath_toth_veto_2012,tarres_toth_valko_2012,toth_valko_2012}, which is a continuum model that shows the same superdiffusive behaviour for $d=1$ and $2$ and diffusive behaviour for $d \geq 3$. We employ the \emph{resolvent method} used in those references, and originally developed in \cite{landim_quastel_salmhoffer_yau_2004,landim_ramirez_yau_2005}, as the main technical tool in this paper, converted appropriately to our lattice-based models. The method  gives us bounds on the Laplace transform of $E$ and proceeds as follows. First the growth of $E(t)$ is connected to the growth  of the compensator of the random walk via a property called Yaglom reversibility (Lemma \ref{Resolvent_Lem_Yaglom}). Next the Laplace transform of the compensator is connected to the symmetric and anti-symmetric parts of the generators of the model dynamics via a variational formula (Equation \ref{eq:Resolvent_VarForm}). The key is that under the symmetrised dynamics, the active particle takes random walk steps independently of the environment of arrows, and hence can be analysed exactly. An upper bound on the Laplace transform of $E$ is obtained by discarding the contributions from the anti-symmetric part in the variational formula (Section \ref{Sect_UpperBound}). Obtaining a lower bound is more challenging and the strategy is to restrict the variational formula to a subspace of linear functionals over which exact computations can be carried out (Section \ref{Sect_LowerBound}).

%%%%%%%%%%%%%%%%%%%%%%%%%% NOTATION %%%%%%%%%%%%%%%%%%%%%%%%%%%%%%
\subsection{Notation}

Before stating our main results, we need some notation. For $d \geq 1$, define $\Z^d_\star = \Z^d \cup \{\star\}$, where $\star$ is an abstract symbol acting as a placeholder for the active particle (the hand). Let $\mathcal{E} = \{ \pm e_1, \dots, \pm e_d\}$ be the canonical unit vectors in $\R^d$ and $\Omega =  \mathcal{E}^{\Z^d_\star}$. The elements of $\Omega$ represent the environment of site arrows \emph{as seen from the position of the walker}, together with the information about the arrow in the walker's hand. The evolution of the environment as seen by the walker is described by a continuous time jump process, $\left(\eta_t\right)_{t \in \mathbb{R}}$, taking values in $\Omega$.

For $\omega \in \Omega$, let $\omega(x)_i$ denote the $i^{\mathrm{th}}$ component of $\omega(x) \in \R^d$. Define the \emph{shift} maps on $\Omega$ by
\[
\tau_{e}\omega(x)=\begin{cases}
\omega(x+e), & \textrm{if }x\neq\star\\
\omega(\star), & \textrm{if }x=\star
\end{cases}
\qquad \textrm{and}\qquad
\tau_{\star}\omega(x)=\begin{cases}
\omega(x+\omega(\star)), & \textrm{if }x\neq\star\\
\omega(\star), & \textrm{if }x=\star
\end{cases}
\]
for $e \in \mathcal{E}$, and the \emph{swap} map
\[
\sigma\omega(x)=\begin{cases}
\omega(x), & \textrm{if }x\neq0,\star\\
\omega(\star), & \textrm{if }x=0\\
\omega(0), & \textrm{if }x=\star.
\end{cases}
\]
The environments seen by the walker after a step in direction $e \in \mathcal{E}$ and a step in the direction of the hand arrow are given by $\tau_e \omega$ and $\tau_\star \omega$, respectively. The state $\sigma \omega$ gives the environment  after the hand and site arrows are exchanged. 

The generator for the motion \emph{step-then-swap} at rate 1 is given by
\[
T[\tau] f( \omega )
	:= f(\sigma  \tau \omega)
		-f(\omega)
\]
where $\tau$ is any one of the step maps above. Likewise the generator for the motion where swaps occur before and after a step with probability $\tfrac{1}{2}$ is
\[
\widetilde{T}[\tau] f( \omega )
	:= \tfrac{1}{4} f(\tau \omega)
		+ \tfrac{1}{4}	f(\sigma  \tau  \omega)
		+ \tfrac{1}{4}	f(  \tau \sigma  \omega)
		+  \tfrac{1}{4}	f(\sigma  \tau \sigma \omega)
		-f(\omega).
\]

%% DEF: model definitions
\begin{defn}[Generators]
\label{intro_Def_Generators}
For $d \geq 1$ the generators for the process $\eta$ in the M1 and M2 models respectively are defined to be
\[
\textrm{M1}: \ G := T[\tau_\star]
	\qquad \textrm{and} \qquad
	\textrm{M2}: \ \widetilde{G} := \widetilde{T}[\tau_\star].
\]
Let $U$ be the generator
\[
U := \frac{1}{2d} \sum_{e \in \mathcal{E}} \widetilde{T}[\tau_e],
\]
then, for fixed $\varepsilon > 0$, the generators for $\textrm{M1}_\varepsilon$ and $\textrm{M2}_\varepsilon$ models are
\[
\textrm{M1}_\varepsilon: \ G_\varepsilon := G + \varepsilon U
	\qquad \textrm{and} \qquad
	\textrm{M2}_\varepsilon: \ \widetilde{G}_\varepsilon := \widetilde{G} + \varepsilon U.
\]
\end{defn}

%% REM: Redundant
\begin{rem}
Technically it is redundant to swap the hand and site arrows before and after a step in the M2 model, since the composition of two random swaps is equal in law to one random swap. The rule is presented in this way to simplify the form of $\widetilde{T}$ when we take its adjoint. 
\end{rem}

To complete the model description we must specify the law of the random initial arrow configuration. We will consider product measures of the form
\begin{equation}
\label{eq:Intro_ProdMeasure}
\pi(d\omega) = \mu(d\omega(\star)) \otimes \bigotimes_{x \in \Z^d} \mu(d\omega(x)),
\end{equation}
with $\mu$ a probability distribution on $\mathcal{E}$. Our interest is in environments that are \emph{drift-free}, and to this end we will define, $\mu_{p}$, to be the  distribution given by
\[
\mu_p(+e_i) = \tfrac{1}{2} p_i = \mu_p(-e_i),
	\qquad \textrm{for } i = 1,2,\dots,d,
\]
where
\[
p \in \mathcal{P} := \Big\{ p \in [0,1]^d : \sum_{i=1}^d p_i = 1 \Big\}.
\]
We will write $\pi_p$ for the product measure that has $\mu = \mu_p$ (see Figure \ref{Intro_Fig_Initial_Conditions}). Notice that in one dimension there is only $p = 1$. When $p$ has all its components equal to $d^{-1}$ we will call $\pi_p$ \emph{isotropic} and otherwise we will call $\pi_p$ \emph{anisotropic}. If $p$ equals one of the canonical unit vectors, then we will call $\pi_p$ \emph{totally anisotropic}. Observe that, since $\varepsilon > 0$, the dynamics under $G_\varepsilon$ and $\widetilde{G}_\varepsilon$ are not constrained to a one-dimensional subspace under the totally anisotropic initial conditions. 

It is essential for our methods that $\pi$ is stationary for the generators in Definition \ref{intro_Def_Generators}.

\begin{figure}
\begin{center}
\begin{tikzpicture}[scale=0.65]
\draw[fill] (-2,-2) circle [radius = 0.1];
\draw[->, line width=0.4mm] (-2,-2) -- (-2,-2+0.4);
\draw[fill] (-2,-1) circle [radius = 0.1];
\draw[->, line width=0.4mm] (-2,-1) -- (-2,-1-0.4);
\draw[fill] (-2,-0) circle [radius = 0.1];
\draw[->, line width=0.4mm] (-2,-0) -- (-2,0-0.4);
\draw[fill] (-2,+1) circle [radius = 0.1];
\draw[->, line width=0.4mm] (-2,+1) -- (-2+0.4,+1);
\draw[fill] (-2,+2) circle [radius = 0.1];
\draw[->, line width=0.4mm] (-2,+2) -- (-2-0.4,+2);

\draw[fill] (-1,-2) circle [radius = 0.1];
\draw[->, line width=0.4mm] (-1,-2) -- (-1+0.4,-2);
\draw[fill] (-1,-1) circle [radius = 0.1];
\draw[->, line width=0.4mm] (-1,-1) -- (-1,-1-0.4);
\draw[fill] (-1,-0) circle [radius = 0.1];
\draw[->, line width=0.4mm] (-1,-0) -- (-1+0.4,0);
\draw[fill] (-1,+1) circle [radius = 0.1];
\draw[->, line width=0.4mm] (-1,+1) -- (-1,+1+0.4);
\draw[fill] (-1,+2) circle [radius = 0.1];
\draw[->, line width=0.4mm] (-1,+2) -- (-1,+2-0.4);

\draw[fill] (0,-2) circle [radius = 0.1];
\draw[->, line width=0.4mm] (0,-2) -- (0-0.4,-2);
\draw[fill] (0,-1) circle [radius = 0.1];
\draw[->, line width=0.4mm] (0,-1) -- (0-0.4,-1);
\draw[fill] (0,-0) circle [radius = 0.1];
\draw[->, line width=0.4mm] (0,-0) -- (0,0+0.4);
\draw[fill] (0,+1) circle [radius = 0.1];
\draw[->, line width=0.4mm] (0,+1) -- (0,+1-0.4);
\draw[fill] (0,+2) circle [radius = 0.1];
\draw[->, line width=0.4mm] (0,+2) -- (0+0.4,+2);

\draw[fill] (1,-2) circle [radius = 0.1];
\draw[->, line width=0.4mm] (1,-2) -- (1+0.4,-2);
\draw[fill] (1,-1) circle [radius = 0.1];
\draw[->, line width=0.4mm] (1,-1) -- (1,-1-0.4);
\draw[fill] (1,-0) circle [radius = 0.1];
\draw[->, line width=0.4mm] (1,-0) -- (1,0+0.4);
\draw[fill] (1,+1) circle [radius = 0.1];
\draw[->, line width=0.4mm] (1,+1) -- (1-0.4,+1);
\draw[fill] (1,+2) circle [radius = 0.1];
\draw[->, line width=0.4mm] (1,+2) -- (1,+2+0.4);

\draw[fill] (2,-2) circle [radius = 0.1];
\draw[->, line width=0.4mm] (2,-2) -- (2,-2+0.4);
\draw[fill] (2,-1) circle [radius = 0.1];
\draw[->, line width=0.4mm] (2,-1) -- (2+0.4,-1);
\draw[fill] (2,-0) circle [radius = 0.1];
\draw[->, line width=0.4mm] (2,-0) -- (2+0.4,0);
\draw[fill] (2,+1) circle [radius = 0.1];
\draw[->, line width=0.4mm] (2,+1) -- (2,+1-0.4);
\draw[fill] (2,+2) circle [radius = 0.1];
\draw[->, line width=0.4mm] (2,+2) -- (2,+2-0.4);

\node at (4.15, 0) { $\sim \pi_{(1/2, 1/2)}$};

\draw[fill, color = red] (0.1,0) circle [ radius = 0.1];
\draw[->, line width=0.4mm, color = red] (0.1,0) -- (0.1+0.4,0);
\end{tikzpicture} $\qquad$
\begin{tikzpicture}[scale=0.65]
\draw[fill] (-2,-2) circle [radius = 0.1];
\draw[->, line width=0.4mm] (-2,-2) -- (-2+0.4,-2);
\draw[fill] (-2,-1) circle [radius = 0.1];
\draw[->, line width=0.4mm] (-2,-1) -- (-2-0.4,-1);
\draw[fill] (-2,-0) circle [radius = 0.1];
\draw[->, line width=0.4mm] (-2,-0) -- (-2-0.4,0);
\draw[fill] (-2,+1) circle [radius = 0.1];
\draw[->, line width=0.4mm] (-2,+1) -- (-2+0.4,+1);
\draw[fill] (-2,+2) circle [radius = 0.1];
\draw[->, line width=0.4mm] (-2,+2) -- (-2-0.4,+2);

\draw[fill] (-1,-2) circle [radius = 0.1];
\draw[->, line width=0.4mm] (-1,-2) -- (-1+0.4,-2);
\draw[fill] (-1,-1) circle [radius = 0.1];
\draw[->, line width=0.4mm] (-1,-1) -- (-1-0.4,-1);
\draw[fill] (-1,-0) circle [radius = 0.1];
\draw[->, line width=0.4mm] (-1,-0) -- (-1+0.4,0);
\draw[fill] (-1,+1) circle [radius = 0.1];
\draw[->, line width=0.4mm] (-1,+1) -- (-1+0.4,+1);
\draw[fill] (-1,+2) circle [radius = 0.1];
\draw[->, line width=0.4mm] (-1,+2) -- (-1-0.4,+2);

\draw[fill] (0,-2) circle [radius = 0.1];
\draw[->, line width=0.4mm] (0,-2) -- (0-0.4,-2);
\draw[fill] (0,-1) circle [radius = 0.1];
\draw[->, line width=0.4mm] (0,-1) -- (0-0.4,-1);
\draw[fill] (0,-0) circle [radius = 0.1];
\draw[->, line width=0.4mm] (0,-0) -- (0+0.4,0);
\draw[fill] (0,+1) circle [radius = 0.1];
\draw[->, line width=0.4mm] (0,+1) -- (0-0.4,+1);
\draw[fill] (0,+2) circle [radius = 0.1];
\draw[->, line width=0.4mm] (0,+2) -- (0+0.4,+2);

\draw[fill] (1,-2) circle [radius = 0.1];
\draw[->, line width=0.4mm] (1,-2) -- (1+0.4,-2);
\draw[fill] (1,-1) circle [radius = 0.1];
\draw[->, line width=0.4mm] (1,-1) -- (1-0.4,-1);
\draw[fill] (1,-0) circle [radius = 0.1];
\draw[->, line width=0.4mm] (1,-0) -- (1+0.4,0);
\draw[fill] (1,+1) circle [radius = 0.1];
\draw[->, line width=0.4mm] (1,+1) -- (1-0.4,+1);
\draw[fill] (1,+2) circle [radius = 0.1];
\draw[->, line width=0.4mm] (1,+2) -- (1+0.4,+2);

\draw[fill] (2,-2) circle [radius = 0.1];
\draw[->, line width=0.4mm] (2,-2) -- (2+0.4,-2);
\draw[fill] (2,-1) circle [radius = 0.1];
\draw[->, line width=0.4mm] (2,-1) -- (2+0.4,-1);
\draw[fill] (2,-0) circle [radius = 0.1];
\draw[->, line width=0.4mm] (2,-0) -- (2+0.4,0);
\draw[fill] (2,+1) circle [radius = 0.1];
\draw[->, line width=0.4mm] (2,+1) -- (2-0.4,+1);
\draw[fill] (2,+2) circle [radius = 0.1];
\draw[->, line width=0.4mm] (2,+2) -- (2-0.4,+2);

\node at (3.75, 0) { $\sim \pi_{(1,0)}$};

\draw[fill, color = red] (0.0,0.15) circle [ radius = 0.1];
\draw[->, line width=0.4mm, color = red] (0.0,0.15) -- (0.0+0.4,0.15);
\end{tikzpicture} 
\caption{\label{Intro_Fig_Initial_Conditions} A typical initialisation from the isotropic measure $\pi_{(1/2, 1/2)}$ (left) and the totally anisotropic measure $\pi_{(1,0)}$ (right). }
\end{center}
\end{figure}
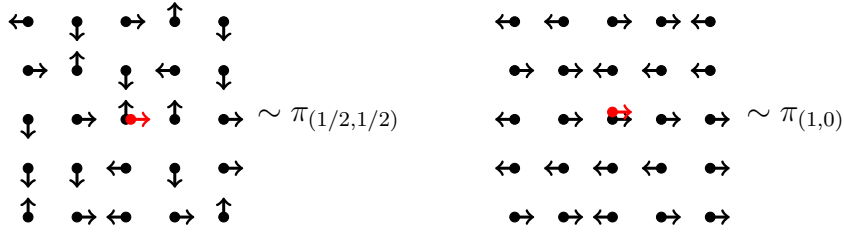

% End of fig_initial_conditions

%% PROP: Stationarity
\begin{prop}[Stationarity]
For every $p \in \mathcal{P}$, the product measure $\pi_p$ is stationary for all of the generators in Definition \ref{intro_Def_Generators}. 
\end{prop}

\begin{proof}
Immediate from the fact $\tau_\star$, $\tau_e$ and $\sigma$ are bijections on $\Omega$ and preserve product measures.  
\end{proof}

\begin{rem}[Conservation of momentum]
 The local sum of the arrows is conserved by the dynamics in both models. This is  analogous to the transport and exchange of momentum in a system of colliding particles such as a hard-sphere gas. 
\end{rem}

\begin{rem}[Entropy production]
 In model M1 the embedded jump chain is completely determined by the initial environment, so the entropy of the dynamics depends on whether the range of the walk grows linearly or sublinearly. In model M2 there is extra randomness at every jump for which $\omega(0) \neq \omega(\star)$, so we would expect the dynamics of model M2 to have positive entropy.
\end{rem}

Our first result characterises the behaviour of $E$ in the simplest case when $d=1$.

%% THM:
\begin{thm}[$d = 1$, M1 and M2]
\label{Intro_Thm_D1_NoEps} 
\begin{enumerate}[(i)]
\item For any deterministic initial configuration of arrows, the dynamics of the M1 model are ballistic. In particular
\[
\liminf_{t \to \infty}\frac{|X_t|}{t} \geq \frac{1}{3},
	\qquad \textrm{with probability }1.
\]

\item With $\pi = \pi_{1}$, the M2 model is the true self-avoiding walk \cite{toth_1995}. In particular we have constants $C, D > 0$ such that 
\[
Ct^{4/3} \leq E(t) \leq Dt^{4/3}. 
\]
\end{enumerate}
\end{thm}

\noindent The first part follows by an easy case-by-case analysis. The second follows from noticing that the one-dimensional environment of arrows can be viewed as the gradient of a local time profile. The proofs are presented in Section \ref{Sect_OneDim}.

From now on we do not analyse $E$ directly, but prove results for the asymptotic singularity  of the Laplace transform
\[
\widehat{E}(\lambda) = \int^\infty_0 E(t)e^{-\lambda t}dt,
\]
as $\lambda \searrow 0$. A consequence of using the resolvent method is that we must settle for proving bounds on the growth rate of $\widehat{E}(\lambda)$, rather than sharp asymptotics. In one-dimension our method only recovers the following bounds.

%% THM:
\begin{thm}[$d = 1$, $\textrm{M1}_\varepsilon$ and $\textrm{M2}_\varepsilon$]
\label{Intro_Thm_D1}
Let $\pi = \pi_{1}$. For both the $\textrm{M1}_\varepsilon$ and $\textrm{M2}_\varepsilon$ models we have constants $C, D > 0$ such that
\[
C\lambda^{-9/4} \leq \widehat{E}(\lambda) \leq D \lambda^{-5/2}.
\]
\end{thm}

%% REM: Discrete time
\begin{rem}[Discrete time]
It is a standard calculation to show the asymptotic behaviour of $\widehat{E}(\lambda)$, and hence the statement of the main results, is exactly the same for the discrete-time version of the models.
\end{rem}

%% REM: Bounds
\begin{rem}[Constants]
The constants in the theorems are not sharp (and depend on $\varepsilon$). The bounds apply for all $\lambda > 0$ sufficiently small.
\end{rem}

In Theorem \ref{Intro_Thm_D1}, the bounds would correspond to 
\[
C t^{5/4} \leq E(t) \leq D t^{3/2}
\]
in the time domain if a Tauberian inversion were possible. This is not the case for the lower bound without further regularity assumptions on $E$, however the upper bound is valid and comes from the relationship
\[
E(t) \leq t^{-1} D'\widehat{E}(t^{-1}),
	\qquad \textrm{for every } t > 0,
\]
where $D'>0$ is a constant. A proof for this follows along the lines of \cite{quastel_valko_2008}. These bounds strictly contain the $t^{4/3}$ growth of $E(t)$ for the one-dimensional TSAW, so neither bound is sharp. Notice that adding the randomising effect $\varepsilon U$ to the generator of $G$ completely destroys the ballistic growth seen in the one-dimensional M1 model.

In three and more dimensions, transience of simple random walk is enough to prove that the walker does not behave superdiffusively. The proof of this, as well as the  upper bounds for $d=1$ and $2$, follows by comparing the system to random walk in random scenery. On the other hand, subdiffusivity is excluded by the property of Yaglom reversibility of the generators (see Lemma \ref{Resolvent_Lem_Yaglom} and \cite{dobrushin_1988,tarres_toth_valko_2012,yaglom_1949}). We therefore conclude diffusive scaling ($E(t) \asymp t$) when $d \geq 3$:

%% THM:
\begin{thm}[$d \geq 3$]
\label{Intro_Thm_D3}
Let $\pi = \pi_p$, for any $p \in \mathcal{P}$. Then for both the $\textrm{M1}_\varepsilon$ and $\textrm{M2}_\varepsilon$ models we have constants $C, D > 0$ such that
\[
C{\lambda}^{-2} \leq \widehat{E}(\lambda) \leq D \lambda^{-2}.
\]
Furthermore, in this case we can conclude that there exists constants $C',D' > 0$ such that
\[
C' t \leq E(t) \leq D' t, \qquad \textrm{for all } t \geq 0 \textrm{ sufficiently large}.
\]
\end{thm}

The two-dimensional case is the hardest to analyse. We are able to prove upper bounds of the order of $t\log t$ for both the $\textrm{M1}_\varepsilon$ and $\textrm{M2}_\varepsilon$ models started from any initial product measure $\pi_p$. In the totally anisotropic case where $\pi_p$ has $p = (1,0)$ or $(0,1)$ --- recall Figure \ref{Intro_Fig_Initial_Conditions} --- a lower bound can be proved of order $t \sqrt{\log t}$.

%% THM:
\begin{thm}[$d = 2$]
\label{Intro_Thm_D2}
\begin{enumerate}[(i)]
\item  Let $\pi = \pi_p$, for any $p \in \mathcal{P}$. Then for both the $\textrm{M1}_\varepsilon$ and $\textrm{M2}_\varepsilon$ models we have a constant $D > 0$ such that
\[
\widehat{E}(\lambda) \leq D \lambda^{-2}\log (\lambda^{-1}).
\]

\item (Totally anisotropic) If $p = (1,0)$ or $(0,1)$ then there exists a constant $C>0$ such that for both the $\textrm{M1}_\varepsilon$ and $\textrm{M2}_\varepsilon$ models
\[
C \lambda^{-2} \sqrt{\log(\lambda^{-1})} \leq \widehat{E}(\lambda).
\]
\end{enumerate}
\end{thm}

\noindent We are unable to obtain a superdiffusive lower bound in the non-totally anisotropic case --- that is, $p \neq (1,0)$ or $(0,1)$ --- however it should be possible to derive a bound of $t \log \log t$ if the computations in \cite{toth_valko_2012} could be replicated. The obstruction to this is that we are unable to analyse a transition kernel of simple random walk on a specific graph in sufficient detail (see the final part of Section \ref{Sect_Completing}). At present all we can say is that $E(t)$ scales at least linearly, by Lemma \ref{Resolvent_Lem_Yaglom}.

In line with \cite{toth_valko_2012} we make the following predictions about the true asymptotic growth of the mean-squared displacement in two dimensions:

%% CONJ: True growth
\begin{conjecture}[$d = 2$]
For both the $\mathrm{M1}_\varepsilon$ and $\mathrm{M2}_\varepsilon$ models, as $t \to \infty$
\begin{align*}
p = (1,0) \ \textrm{or} \ (0,1)  \ &: \ E(t) \asymp t(\log t)^{2/3}\\
p \neq (1,0) \ \textrm{or} \ (0,1) \ &: \ E(t) \asymp t(\log t)^{1/2}.
\end{align*}
\end{conjecture}

\noindent Although, based on Theorem \ref{Intro_Thm_D1_NoEps}, we might expect the $\mathrm{M1}_\varepsilon$ model to be more superdiffusive than the $\mathrm{M2}_\varepsilon$ model, the fact that the asymptotic bounds in Theorems \ref{Intro_Thm_D1}, \ref{Intro_Thm_D3} and \ref{Intro_Thm_D2} and their subsequent proofs are insensitive to the choice of models suggests that the true asymptotic rates ought to agree. The exponents in this conjecture are derived in \cite{toth_valko_2012} using the non-rigorous Alder--Wainwright scaling argument \cite{alder_wainwright_1967,alder_wainwright_1970,forster_1970}. Notice that we expect the logarithmic exponent to differ from the isotropic case only in the completely anisotropic case. Ideally we would also like to prove bounds for the original version of the models where $\varepsilon = 0$. The obstruction there is the lack of ellipticity that allows us to study the process under the much simpler dynamics of $U$ in Lemmas \ref{Upper_Lem_Coupling} and \ref{Lower_Lem_Correlation}.

\subsubsection*{Paper overview}

In Section \ref{Sect_OneDim} we use simple stand-alone arguments to prove Theorem \ref{Intro_Thm_D1_NoEps}. In Section \ref{Sect_Resolvent} we introduce the resolvent method, which we use in Section \ref{Sect_UpperBound} to produce the upper lower bounds on the mean-square displacement. In Section \ref{Sect_LowerBound} we use the resolvent method to produce lower bounds that are valid for $d=1$ and the totally anisotropic case in $d=2$. Finally, in Section \ref{Sect_Completing} we gather all these results together to present proofs of Theorems \ref{Intro_Thm_D1}, \ref{Intro_Thm_D3} and \ref{Intro_Thm_D2}.

%%%%%%%%%%%%%%%%%%%%%%%%%%%%%%%%%%%%%%%%%%%%%%%%%%%%%
%%%%%%%%%%%%%%%%%%%%%%%%%%%%%%%%%%%%%%%%%%%%%%%%%%%%%
%%%%%%%%%%%%     2. D = 1              %%%%%%%%%%%%%%
%%%%%%%%%%%%%%%%%%%%%%%%%%%%%%%%%%%%%%%%%%%%%%%%%%%%%
%%%%%%%%%%%%%%%%%%%%%%%%%%%%%%%%%%%%%%%%%%%%%%%%%%%%%
\section{One-dimension; Proof of Theorem \ref{Intro_Thm_D1_NoEps}}
\label{Sect_OneDim}

%% Proof of ballistic case using easy argument, show M2 is TSAW
				 %% and discussion of scaling
%%%%%%%%%%%%%%%%%%%%%%%%%%%%%%%%%%%%%%%%%%%%%%%%%%%%%%%%%%%%%%%%
%%%%%%%%%%%%%%%%%%%%%%%%%%%%%%%%%%%%%%%%%%%%%%%%%%%%%%%%%%%%%%%%
%%%%%%%%%%%%%%%%     ONE DIMENSION         %%%%%%%%%%%%%%%%%%%%%
%%%%%%%%%%%%%%%%%%%%%%%%%%%%%%%%%%%%%%%%%%%%%%%%%%%%%%%%%%%%%%%%
%%%%%%%%%%%%%%%%%%%%%%%%%%%%%%%%%%%%%%%%%%%%%%%%%%%%%%%%%%%%%%%%

In this section we prove Theorem \ref{Intro_Thm_D1_NoEps}, first by showing the M1 model is ballistic through a simple case-by-case argument and second by showing that the M2 model is equivalent to the true self-avoiding walk (TSAW) \cite{toth_1995}.

\begin{proof}[Proof for the M1 model]
Let $Y_n$ denote the position of the (discrete-time) walk after $n$ steps (i.e.~$Y_n = X_{T_n}$, where $T_n$ is the $n^\mathrm{th}$ jump time). The key is to notice that if ever the hand and current site arrows point in the same direction, then this situation is restored after at most three steps, during which time the walker moves one unit in that original direction. To see this, there are two cases to consider. First, if the arrow at the next site also agrees with the hand and site arrow then we have the sequence
\[
\begin{array}{cccccc}
\rightarrow & \rightarrow &  &  & \rightarrow & \rightarrow\\
\rightarrow &  &  &  &  & \rightarrow
\end{array}
\]
which takes one step. Likewise if the next site arrow disagrees with the hand and site arrow then we have the sequence
\[
\begin{array}{cccccccccccccc}
\rightarrow & \leftarrow &  &  & \rightarrow & \rightarrow &  &  & \leftarrow & \rightarrow &  &  & \leftarrow & \rightarrow\\
\rightarrow &  &  &  &  & \leftarrow &  &  & \rightarrow &  &  &  &  & \rightarrow
\end{array}
\]
which takes three steps. In both cases we have gone from state $\rightrightarrows$ to itself in at most three steps.

It remains to notice that we must eventually reach the state where both hand and site arrows agree. This is clear because, up to symmetry, the only initial states where this could be avoided are
\[
\begin{array}{cc}
\leftarrow & \rightarrow\\
\rightarrow
\end{array}\textrm{\qquad and\qquad}\begin{array}{cc}
\leftarrow & \leftarrow\\
\rightarrow
\end{array}.
\]
The first leads immediately to the persistent state and the second does so after two steps:
\[
\begin{array}{cccccccccc}
\leftarrow & \leftarrow &  &  & \leftarrow & \rightarrow &  &  & \leftarrow & \rightarrow\\
\rightarrow &  &  &  &  & \leftarrow &  &  & \leftarrow
\end{array}.
\]

We can now conclude $|Y_n| \geq \lfloor \tfrac{1}{3} (n - 2) \rfloor$. Since $T_n / n \to 1$, we have
\[
\liminf_{t \to \infty} \frac{|X_t|}{t}
	= \liminf_{n \to \infty} \frac{|Y_n|}{T_n}
	=  \liminf_{n \to \infty} \frac{|Y_n|}{n} \cdot \frac{n}{T_n}
	\geq \frac{1}{3},
\]
with probability one.

\end{proof}

\begin{proof}[Proof for the M2 model]
We begin by constructing a function $\ell_0$ on the edges of the one-dimensional lattice whose negative gradient is equal to the configuration of the arrows. We will show that this function evolves in the same way as the local time profile of the one-dimensional TSAW with nearest-neighbour interaction. To do so, let $\gamma_t$ denote the environment of arrows with respect to a fixed origin (i.e.~not with respect to the position of the walker). Also define $\gamma_t(X_t)$, where $X_t$ is the position of the walker, to equal the sum of the hand and site arrows at that location. Therefore $\gamma_t(x) \in \{-1,+1\}$ if $x \neq X_t$ and $\gamma_t(X_t) \in \{-2,0,+2\}$ (see Figure \ref{OneDim_Fig_TSAW}).

For edges $(x, x+1)$ with $x \in \mathbb{Z}$, recursively define $\ell_0$ such that
\[
\nabla \ell_0(x) := \ell_0(x, x+1) - \ell_0(x-1, x) = -\gamma_0(x),
\]  
with the arbitrary choice $\ell_0(0,+1) = 0$ (see Figure \ref{OneDim_Fig_TSAW}). For $t \geq 1$, define 
\[
\ell_{t}(x,x+1)=\begin{cases}
\ell_{t-1}(x,x+1)+1, & \textrm{if }(X_{t-1},X_{t})=(x,x+1)\textrm{ or }(x+1,x)\\
\ell_{t-1}(x,x+1), & \textrm{otherwise}.
\end{cases}
\]
Therefore $\ell_t(x, x+1)$ counts the number of (unsigned) crossings of the edge $(x,x+1)$. With this rule it is straightforward to see that the relationship $\nabla \ell_t = - \gamma_t$ is preserved, because, for example, if $\nabla \ell_{t-1} = - \gamma_{t-1}$, $X_{t-1} = x$ and $X_{t+1} = x+1$, then we must have $\gamma_{t}(x) = \gamma_{t-1}(x) - 1$ and $\gamma_{t}(x+1) = \gamma_{t-1}(x+1) + 1$ whilst 
\[
\nabla \ell_t(x) = \nabla \ell_{t-1}(x) + 1=  - \gamma_{t-1}(x) + 1,
	\quad \nabla \ell_t(x+1) = \nabla \ell_{t-1}(x+1) - 1=  - \gamma_{t-1}(x+1) - 1.
\]
Since $X_t$ is forced to move right if $\nabla \ell_t(X_t) = -2$ (i.e.~$\gamma_t(X_t) = +2$, so the hand and site arrows both point right) and takes a uniform random step if  $\nabla \ell_t(X_t) = 0$ (i.e.~$\gamma_t(X_t) = 0$, so the hand and site arrows point in opposite directions), we see that $X$ evolves according to the rule
\[
\mathbb{P}(X_{t+1}=x\pm1|X_{t}=x)=\begin{cases}
1/2, & \textrm{if }\nabla\ell_{t}(x)=0\\
1, & \textrm{if }\nabla\ell_{t}(x)=\mp2\\
0, & \textrm{if }\nabla\ell_{t}(x)=\pm2
\end{cases}
\]
which is precisely the law of the TSAW in one dimension. It now follows from \cite{toth_1995} that $E(t) = \Theta(t^{4/3})$ as $t \to \infty$. 
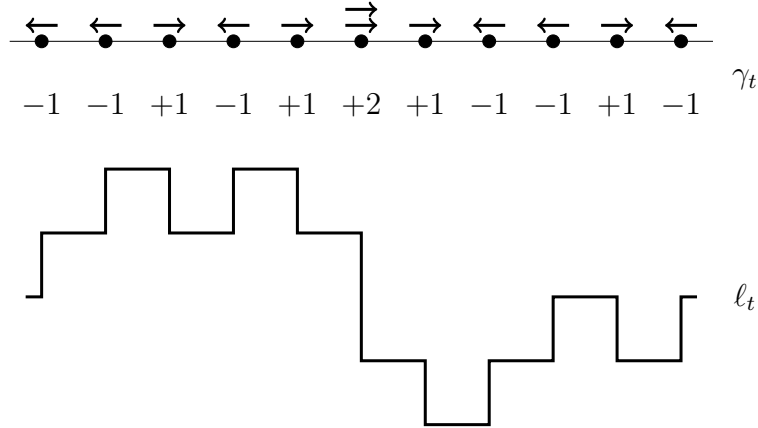
\begin{figure}
\begin{center}
\begin{tikzpicture}[scale=0.85]

\draw[fill] (-5,0) circle [radius = 0.1];
\draw[fill] (-4,0) circle [radius = 0.1];
\draw[fill] (-3,0) circle [radius = 0.1];
\draw[fill] (-2,0) circle [radius = 0.1];
\draw[fill] (-1,0) circle [radius = 0.1];
\draw[fill] (0,0) circle [radius = 0.1];
\draw[fill] (+1,0) circle [radius = 0.1];
\draw[fill] (+2,0) circle [radius = 0.1];
\draw[fill] (+3,0) circle [radius = 0.1];
\draw[fill] (+4,0) circle [radius = 0.1];
\draw[fill] (+5,0) circle [radius = 0.1];

\draw[-, line width=0.1mm] (-5.5, 0) -- (5.5,0);

\draw[<-, line width=0.4mm] (-5.25, 0.25) -- (-4.75,0.25);
\draw[<-, line width=0.4mm] (-4.25, 0.25) -- (-3.75,0.25);
\draw[->, line width=0.4mm] (-3.25, 0.25) -- (-2.75,0.25);
\draw[<-, line width=0.4mm] (-2.25, 0.25) -- (-1.75,0.25);
\draw[->, line width=0.4mm] (-1.25, 0.25) -- (-0.75,0.25);
\draw[->, line width=0.4mm] (-0.25, 0.25) -- (0.25,0.25);
\draw[->, line width=0.4mm] (0.75, 0.25) -- (1.25,0.25);
\draw[<-, line width=0.4mm] (1.75, 0.25) -- (2.25,0.25);
\draw[<-, line width=0.4mm] (2.75, 0.25) -- (3.25,0.25);
\draw[->, line width=0.4mm] (3.75, 0.25) -- (4.25,0.25);
\draw[<-, line width=0.4mm] (4.75, 0.25) -- (5.25,0.25);

\draw[->, line width=0.4mm] (-0.25, 0.5) -- (0.25,0.5);

\node at (-5,-1.0) {$-1$}; 
\node at (-4,-1.0) {$-1$}; 
\node at (-3,-1.0) {$+1$}; 
\node at (-2,-1.0) {$-1$}; 
\node at (-1,-1.0) {$+1$}; 
\node at (0,-1.0) {$+2$}; 
\node at (1,-1.0) {$+1$}; 
\node at (2,-1.0) {$-1$}; 
\node at (3,-1.0) {$-1$}; 
\node at (4,-1.0) {$+1$}; 
\node at (5,-1.0) {$-1$};

\draw[-, line width=0.4mm] (-5.25,-4)--(-5,-4) -- (-5,-3) -- (-4,-3) -- (-4,-2) -- (-3,-2) -- (-3,-3) -- (-2,-3) -- (-2,-2) -- (-1, -2) -- (-1,-3) -- (0,-3) -- (0,-5) -- (1,-5)--(1,-6)--(2,-6)--(2,-5)--(3,-5)--(3,-4)--(4,-4)--(4,-5)--(5,-5)--(5,-4)--(5.25,-4);

\node at (6, -0.6) {$\gamma_t$};
\node at (6, -4) {$\ell_t$};

\end{tikzpicture}
\caption{\label{OneDim_Fig_TSAW} Example of a profile of arrows, $\gamma_t$, at some time $t$ and a local time profile, $\ell_t$, that satisfy $\nabla \ell_t = - \gamma_t$. The random walker is at the unique site $x$ satisfying $\gamma_t(x) \in \{-2,0,+2 \}$.    }
\end{center}
\end{figure}
% End of fig_TSAW
\end{proof}

%%%%%%%%%%%%%%%%%%%%%%%%%%%%%%%%%%%%%%%%%%%%%%%%%%%%%
%%%%%%%%%%%%%%%%%%%%%%%%%%%%%%%%%%%%%%%%%%%%%%%%%%%%%
%%%%%%%%%%%%     3. RESOLVENT METHOD       %%%%%%%%%%
%%%%%%%%%%%%%%%%%%%%%%%%%%%%%%%%%%%%%%%%%%%%%%%%%%%%%
%%%%%%%%%%%%%%%%%%%%%%%%%%%%%%%%%%%%%%%%%%%%%%%%%%%%%
\section{The resolvent method}
\label{Sect_Resolvent}

%% Introduction to method and symm. and antisymm. operators

The key to the proof of Theorems \ref{Intro_Thm_D1}, \ref{Intro_Thm_D3} and \ref{Intro_Thm_D2} is to express $\widehat{E}(\lambda)$ in terms of a variational formula for the resolvent of $G_\varepsilon$ and $\widetilde{G}_\varepsilon$. Throughout, we will be working on the space $L^2(\Omega, \pi)$, for which we will denote the inner product by $(\cdot, \cdot)_\pi$, with $\pi$ any of the product measures from (\ref{eq:Intro_ProdMeasure}). 

As we are about to introduce several extra operators, it will be helpful to use indices to denote which operators are under consideration. Therefore we will write the mean-square displacement as
\[
E_{G_\varepsilon}(t)
	:= \mathbb{E}_{G_\varepsilon}[|X_t|^2]
	\qquad \textrm{and} \qquad
	E_{\widetilde{G}_\varepsilon}(t)
	:= \mathbb{E}_{\widetilde{G}_\varepsilon}[|X_t|^2],
\]
where $\mathbb{P}_V$ denotes that the relevant random process has the dynamics given by generator $V$. 

The invariance of $\pi$ under $\tau_\star$, $\tau_e$ and $\sigma$ immediately gives that we have the following adjoints with respect to $L^2(\Omega, \pi)$ and its inner product $\langle \cdot, \cdot \rangle_\pi$:
\[
G^\dagger f(\omega)  = f(\tau_\star^{-1} \sigma \omega) - f(\omega)
\qquad \textrm{and} \qquad
\widetilde{G}^\dagger = \widetilde{T}[\tau_\star^{-1}],
\]
where we note that 
\[
\tau_{\star}^{-1}\omega(x)=\begin{cases}
\omega(x-\omega(\star)), & \textrm{if }x\neq\star\\
\omega(\star), & \textrm{if }x=\star
\end{cases}.
\]
Also, it is clear that $U$ is self-adjoint, so 
\[
G_\varepsilon^\dagger = G^\dagger + \varepsilon U
	\qquad \textrm{and} \qquad
	\widetilde{G}_\varepsilon^\dagger = \widetilde{G}^\dagger + \varepsilon U.
\]

Our first observation is that, although $G_\varepsilon$ and $\widetilde{G}_\varepsilon$ are not reversible, they are Yaglom reversible, which allows to study $E_{G_\varepsilon}(t)$ and $E_{\widetilde{G}_\varepsilon}(t)$ through the mean-square displacement of the compensator of the corresponding random walks.

%% LEM: Yaglom
\begin{lem}
\label{Resolvent_Lem_Yaglom}
Let $\phi, \widetilde{\phi}:\Omega \to \mathbb{R}^d$ be the compensators of the random walks generated by $G_\varepsilon$ and $\widetilde{G}_\varepsilon$:
\[
\mathrm{M1}_\varepsilon : \ \phi(\omega) = \omega(\star),
\qquad
\mathrm{M2}_\varepsilon : \ \widetilde{\phi}(\omega) =  \tfrac{1}{2}(\omega(0) + \omega(\star)). 
\]
For a generator $V$ and $f : \Omega \to \mathbb{R}$, define 
\[
\Lambda_V(t; f) = \mathbb{E}_V\Big[ \Big( \int_0^t f(\eta_s) ds \Big)^2 \Big],
\]
where $(\eta_s)_{s \geq 0}$ follows the dynamics of $V$. Then we have 
\[
E_{G_\varepsilon}(t) = t + \sum_{i=1}^d \Lambda_{G_\varepsilon}(t; \phi_i),
	\qquad 
	\widehat{E}_{G_\varepsilon}(\lambda) = \lambda^{-2} + \sum_{i=1}^d \widehat{\Lambda}_{G_\varepsilon}(\lambda; \phi_i),
\]
and likewise for $\widetilde{G}_\varepsilon$ and $\widetilde{\phi}$.
\end{lem}

\begin{proof}
We follow the proof in \cite[Lem.~3]{horvath_toth_veto_2012}. Consider the $M2_{\varepsilon}$ model first. Decompose the position of the walker as 
\[
X_t - X_s = M_{s,t} + \int^t_s \widetilde{\phi}(\eta_u) du,
	\qquad M_{s,t} := B_t - B_s,
\]
where $(B_t)_{t \geq 0}$ is a martingale. The square-expectation is then
\[
\mathbb{E}[(X_t - X_s)^2 ]
	= t-s + 2 \, \mathrm{cov \,}\Big( M_{s,t} \, , \int_s^t \widetilde{\phi}(\eta_u) du \Big)
	+ \mathbb{E}\Big[ \Big( \int^t_s \widetilde{\phi}(\eta_u) du\Big)^2 \Big],
\]
and so we have the result if the covariance term vanishes.

The strategy is to prove that $s \mapsto M(s,t)$ is a backward martingale with respect to $\{ \mathcal{F}_{[s,+\infty)} \}_{s \leq t}$, for fixed $t$. With this in place the covariance becomes
\begin{align*}
\mathrm{cov}\Big( M_{s,t} \, , \int_s^t \widetilde{\phi}(\eta_u) du \Big)
	&= \int^t_s \mathbb{E} [M_{s,t} \widetilde{\phi}(\eta_u) ] du \\
	&= \int^t_s \mathbb{E} [ \mathbb{E} [ M_{s,t} \widetilde{\phi}(\eta_u) | \mathcal{F}_{[0,u]} ] ] du \\
	&= \int^t_s \mathbb{E} [M_{s,u} \widetilde{\phi}(\eta_u) ] du \\
	&= \int^t_s \mathbb{E} [ \mathbb{E} [ M_{s,u} \widetilde{\phi}(\eta_u) | \mathcal{F}_{[u,+\infty)} ] ] du
	=  \int^t_s \mathbb{E} [M_{u,u}  \widetilde{\phi}(\eta_u) ] du = 0,
\end{align*} 
where in the second line we have used that $t \mapsto M_{s,t}$ is a forward martingale, by construction.

To show we have a backward martingale, define $\bar{\eta}_t := -\eta_{-t} $  and 
\[
Jf(\omega) := f(-\omega).
\]
Observe that $\widetilde{G}_\varepsilon^\dagger = J \widetilde{G}_\varepsilon J$, since $\tau_\star(-\omega)= -\tau_\star^{-1} \omega$, and therefore $\bar{\eta}$ and $\eta$ are identical in law. So if we write $\bar{X}_t := -X_{-t}$ then
\[
\lim_{h \to 0} h^{-1} \mathbb{E}[X_{s-h} - X_s | \mathcal{F}_{[s,+\infty)}]
	=- \lim_{h \to 0} h^{-1} \mathbb{E}[\bar{X}_{-s+h} - \bar{X}_{-s} | \mathcal{F}_{[s,+\infty)}]
	= -\widetilde{\phi}(\bar{\eta}_{-s})
	= \phi(\eta_s),
\]
as required.

There is a slight complication for the $\mathrm{M1}_\varepsilon$ model. Here, define the operator
\[
Pf(\omega) := f(\sigma \omega),
\]
then we see that $G_\varepsilon^\dagger = PJ G_\varepsilon JP$. Therefore if $\eta$ has the dynamics of $G_\varepsilon$, then $t \mapsto \eta_t$ and $t \mapsto \gamma_t := -\sigma\eta_{-t}$ have the same law. This does not alter the argument above, however, since the random walk corresponding to $\gamma$ still has compensator $-\phi(\eta_{-t})$, so we conclude the result.
\end{proof}

An immediate corollary of Lemma \ref{Resolvent_Lem_Yaglom} is that in all dimensions $E(t) \geq t$. In particular, this proves the lower bound in Theorem \ref{Intro_Thm_D3}.

To state our main tool for the proofs, we must introduce the \emph{symmetric part} of the generators
\begin{align*}
S_\varepsilon 
	&:= \tfrac{1}{2}(G_\varepsilon + G^\dagger_\varepsilon),
	\qquad
	S_\varepsilon  f(\omega)
		= \tfrac{1}{2} f(\sigma \tau_\star \omega) + \tfrac{1}{2} f( \tau_\star^{-1} \sigma \omega) - f(\omega) + \varepsilon Uf(\omega),
\end{align*} 
and the \emph{anti-symmetric part}
\begin{align*}
A &:= \tfrac{1}{2}(G_\varepsilon - G^\dagger_\varepsilon),
	\qquad
	Af(\omega) 
		= \tfrac{1}{2} f(\sigma \tau_\star \omega) - \tfrac{1}{2} f( \tau_\star^{-1} \sigma \omega).
\end{align*} 
For generator $G_\varepsilon$ we have the variational formula
\begin{align}
\label{eq:Resolvent_VarForm}
\widehat{\Lambda}_{G_\varepsilon}(\lambda; \phi_i)
	&= 2\lambda^{-2} \langle\phi_i, (\lambda - G_\varepsilon)^{-1} \phi_i\rangle_\pi \\
	&=  2\lambda^{-2} \sup_{\psi \in L^2(\Omega, \pi)} \Big\{ 2\langle\phi_i, \psi\rangle_\pi - \langle\psi, (\lambda - S_\varepsilon)\psi\rangle_\pi - \langle A\psi, (\lambda - S_\varepsilon)^{-1} A \psi\rangle_\pi \Big\}, \nonumber
\end{align}
for $i = 1,2,\dots,d$. For generator $\widetilde{G}_\varepsilon$, we have exactly the same definitions and variational formula, and we will denote the symmetric and anti-symmetric parts by $\widetilde{S}_\varepsilon$ and $\widetilde{A}$. See Sethuraman \cite{sethuraman} for the derivation of this variational formula.

%%%%%%%%%%%%%%%%%%%%%%%%%%%%%%%%%%%%%%%%%%%%%%%%%%%%%
%%%%%%%%%%%%%%%%%%%%%%%%%%%%%%%%%%%%%%%%%%%%%%%%%%%%%
%%%%%%%%%%%%     4. THE UPPER BOUND        %%%%%%%%%%
%%%%%%%%%%%%%%%%%%%%%%%%%%%%%%%%%%%%%%%%%%%%%%%%%%%%%
%%%%%%%%%%%%%%%%%%%%%%%%%%%%%%%%%%%%%%%%%%%%%%%%%%%%%
\section{The upper bound}
\label{Sect_UpperBound}

%% Drop A part, coupling argument, final integral, Fourier notation 

In this section we derive upper bounds on $\widehat{\Lambda}(\lambda, \phi_i)$ and $\widehat{\Lambda}(\lambda, \widetilde{\phi}_i)$ that are valid for all $d \geq 1$. Unless stated otherwise, all results in this section apply to both $G_\varepsilon$ and $\widetilde{G}_\varepsilon$. Since the results in this section do not depend on the choice of $i \in \{1,2,\dots,d\}$, we will drop the index and write $\phi(\omega)$ and $\widetilde{\phi}(\omega)$ in place of $\phi(\omega)_i$ and $\widetilde{\phi}(\omega)_i$.  

Begin by noticing that the final term in (\ref{eq:Resolvent_VarForm})
\[
2\lambda^{-2} \langle A\psi, (\lambda - S_\varepsilon)^{-1} A\psi\rangle_\pi
\]
is equal to $\widehat{\Lambda}_{S_\varepsilon}(\lambda; A\psi) = 2\lambda^{-2} \widehat{\Phi}_{S_\varepsilon}(\lambda; A\psi)$, where we denote the autocorrelation by
\begin{align}
\label{eq:Lower_Psi}
\Phi_V(t;f) &:= \mathbb{E}_{V}[f(\eta_0)f(\eta_t)],\\
\textrm{so that } \ \Lambda_{V}(t; f) 
	&= \mathbb{E}_{V} \Big[ \Big( \int^t_0 f(\eta_s) ds \Big)^2 \Big]
	= 2\int^t_0 (t-s)\Phi_V(s;f)ds.  \nonumber
\end{align}
Since $\left(\lambda - S_\epsilon\right)$ is positive definite, the final term in (\ref{eq:Resolvent_VarForm}) is non-positive, so dropping it gives the upper bound
\begin{align*}
\widehat{\Lambda}_{G_\varepsilon}(\lambda; \phi) 
	&\leq 2 \lambda^{-2} \sup_{\psi \in L^2(\Omega,\pi)} \Big\{ 2\langle\phi,\psi\rangle_\pi- \langle\psi, (\lambda - S_\varepsilon) \psi\rangle_\pi \Big\}
	= 2\lambda^{-2} \langle\phi, (\lambda - S_\varepsilon)^{-1} \phi\rangle_\pi.
\end{align*}
Notice, however, that $S_\varepsilon \leq \varepsilon U$, and so
\begin{align}
\label{eq:Upper_CorrelationBound}
\widehat{\Lambda}_{G_\varepsilon}(\lambda; \phi)  \nonumber
	\leq 2\lambda^{-2} \langle\phi, (\lambda - \varepsilon U)^{-1} \phi\rangle_\pi
	&= 2 \varepsilon^{-1}\lambda^{-2} \langle\phi, (\varepsilon^{-1}\lambda - U)^{-1} \phi\rangle_\pi \\
	&= 2\varepsilon^{-1}\lambda^{-2}  \widehat{\Phi}_U(\varepsilon^{-1} \lambda; \phi)
\end{align}
and likewise 
\[
\widehat{\Lambda}_{\widetilde{G}_\varepsilon}(\lambda; \widetilde{\phi})
	\leq 2\varepsilon^{-1}\lambda^{-2} \widehat{\Phi}_U(\varepsilon^{-1} \lambda; \widetilde{\phi}).
\]

These upper bounds are now a substantial help as we only need to analyse the model under the dynamics of $U$, which do not depend on the state of the environment of arrows. The following coupling argument allows us to express, $\Phi_U(t; \phi)$, as the return probability of simple random walk with a sticky origin. This enables an exact calculation of $\widehat{\Phi}_U(\lambda; \phi)$.

%% LEM: Coupling 2d
\begin{lem}[A coupling]
\label{Upper_Lem_Coupling}
Let $W$ be a continuous-time rate 1 simple random walk on $\Z^d$ with a sticky origin. That is, when $W=0$, at the next jump time it remains there with probability $\tfrac{1}{2}$, otherwise it takes a simple random walk step. Then
\[
\Phi_U(t; \phi) = d^{-1}\P(W_t = 0 | W_0 = 0) = 2\Phi_U(t; \widetilde{\phi}).
\]
\end{lem}

\begin{proof}
We will only present the proof for $\phi$, since the proof for  $\widetilde{\phi}$ is almost identical. 

Begin by conditioning on the initial profile
\[
\Phi_U(t; \phi) = \int_{\Omega} \omega_1(\star) \rho_t(\omega) \pi(d\omega), \qquad
	\textrm{where } \rho_t(\omega):= \E_U[\eta_t(\star)_1|\eta_0 = \omega].
\]
Introduce the notation
\[
\omega^{+}(x) 
	= \begin{cases}
e_1 & \textrm{if } x = \star, \\
\omega(x) & \textrm{if }x \in \Z^d,
\end{cases}
\qquad\qquad
\omega^{-}(x) 
	= \begin{cases}
-e_1 & \textrm{if } x = \star, \\
\omega(x) & \textrm{if }x \in \Z^d,
\end{cases}	
\]
then by pairing up $\omega \in \Omega$ that agree everywhere except at $x = \star$ we have
\[
\Phi_U(t; \phi) = \frac{1}{2d} \int_{\Omega} \Big( \rho_t(\omega^+) - \rho_t(\omega^-)\Big) \pi(d\omega).
\]
Note that only those $\omega$ with $\omega_1(\star)\neq 0$ contribute to $\Phi_U(t; \phi)$. (See Figure \ref{Upper_Fig_Pairing}.)

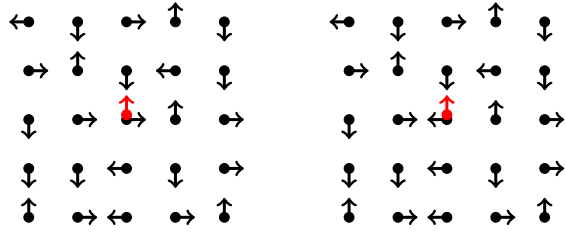
\begin{figure}
\begin{center}
\begin{tikzpicture}[scale=0.65]
\draw[fill] (-2,-2) circle [radius = 0.1];
\draw[->, line width=0.4mm] (-2,-2) -- (-2,-2+0.4);
\draw[fill] (-2,-1) circle [radius = 0.1];
\draw[->, line width=0.4mm] (-2,-1) -- (-2,-1-0.4);
\draw[fill] (-2,-0) circle [radius = 0.1];
\draw[->, line width=0.4mm] (-2,-0) -- (-2,0-0.4);
\draw[fill] (-2,+1) circle [radius = 0.1];
\draw[->, line width=0.4mm] (-2,+1) -- (-2+0.4,+1);
\draw[fill] (-2,+2) circle [radius = 0.1];
\draw[->, line width=0.4mm] (-2,+2) -- (-2-0.4,+2);

\draw[fill] (-1,-2) circle [radius = 0.1];
\draw[->, line width=0.4mm] (-1,-2) -- (-1+0.4,-2);
\draw[fill] (-1,-1) circle [radius = 0.1];
\draw[->, line width=0.4mm] (-1,-1) -- (-1,-1-0.4);
\draw[fill] (-1,-0) circle [radius = 0.1];
\draw[->, line width=0.4mm] (-1,-0) -- (-1+0.4,0);
\draw[fill] (-1,+1) circle [radius = 0.1];
\draw[->, line width=0.4mm] (-1,+1) -- (-1,+1+0.4);
\draw[fill] (-1,+2) circle [radius = 0.1];
\draw[->, line width=0.4mm] (-1,+2) -- (-1,+2-0.4);

\draw[fill] (0,-2) circle [radius = 0.1];
\draw[->, line width=0.4mm] (0,-2) -- (0-0.4,-2);
\draw[fill] (0,-1) circle [radius = 0.1];
\draw[->, line width=0.4mm] (0,-1) -- (0-0.4,-1);
\draw[fill] (0,-0) circle [radius = 0.1];
\draw[->, line width=0.4mm] (0,-0) -- (0+0.4,0);
\draw[fill] (0,+1) circle [radius = 0.1];
\draw[->, line width=0.4mm] (0,+1) -- (0,+1-0.4);
\draw[fill] (0,+2) circle [radius = 0.1];
\draw[->, line width=0.4mm] (0,+2) -- (0+0.4,+2);

\draw[fill] (1,-2) circle [radius = 0.1];
\draw[->, line width=0.4mm] (1,-2) -- (1+0.4,-2);
\draw[fill] (1,-1) circle [radius = 0.1];
\draw[->, line width=0.4mm] (1,-1) -- (1,-1-0.4);
\draw[fill] (1,-0) circle [radius = 0.1];
\draw[->, line width=0.4mm] (1,-0) -- (1,0+0.4);
\draw[fill] (1,+1) circle [radius = 0.1];
\draw[->, line width=0.4mm] (1,+1) -- (1-0.4,+1);
\draw[fill] (1,+2) circle [radius = 0.1];
\draw[->, line width=0.4mm] (1,+2) -- (1,+2+0.4);

\draw[fill] (2,-2) circle [radius = 0.1];
\draw[->, line width=0.4mm] (2,-2) -- (2,-2+0.4);
\draw[fill] (2,-1) circle [radius = 0.1];
\draw[->, line width=0.4mm] (2,-1) -- (2+0.4,-1);
\draw[fill] (2,-0) circle [radius = 0.1];
\draw[->, line width=0.4mm] (2,-0) -- (2+0.4,0);
\draw[fill] (2,+1) circle [radius = 0.1];
\draw[->, line width=0.4mm] (2,+1) -- (2,+1-0.4);
\draw[fill] (2,+2) circle [radius = 0.1];
\draw[->, line width=0.4mm] (2,+2) -- (2,+2-0.4);

\draw[fill, color = red] (0,0.1) circle [ radius = 0.1];
\draw[->, line width=0.4mm, color = red] (0,0.1) -- (0,0.1+0.4);
\end{tikzpicture} $\qquad$
\begin{tikzpicture}[scale=0.65]
\draw[fill] (-2,-2) circle [radius = 0.1];
\draw[->, line width=0.4mm] (-2,-2) -- (-2,-2+0.4);
\draw[fill] (-2,-1) circle [radius = 0.1];
\draw[->, line width=0.4mm] (-2,-1) -- (-2,-1-0.4);
\draw[fill] (-2,-0) circle [radius = 0.1];
\draw[->, line width=0.4mm] (-2,-0) -- (-2,0-0.4);
\draw[fill] (-2,+1) circle [radius = 0.1];
\draw[->, line width=0.4mm] (-2,+1) -- (-2+0.4,+1);
\draw[fill] (-2,+2) circle [radius = 0.1];
\draw[->, line width=0.4mm] (-2,+2) -- (-2-0.4,+2);

\draw[fill] (-1,-2) circle [radius = 0.1];
\draw[->, line width=0.4mm] (-1,-2) -- (-1+0.4,-2);
\draw[fill] (-1,-1) circle [radius = 0.1];
\draw[->, line width=0.4mm] (-1,-1) -- (-1,-1-0.4);
\draw[fill] (-1,-0) circle [radius = 0.1];
\draw[->, line width=0.4mm] (-1,-0) -- (-1+0.4,0);
\draw[fill] (-1,+1) circle [radius = 0.1];
\draw[->, line width=0.4mm] (-1,+1) -- (-1,+1+0.4);
\draw[fill] (-1,+2) circle [radius = 0.1];
\draw[->, line width=0.4mm] (-1,+2) -- (-1,+2-0.4);

\draw[fill] (0,-2) circle [radius = 0.1];
\draw[->, line width=0.4mm] (0,-2) -- (0-0.4,-2);
\draw[fill] (0,-1) circle [radius = 0.1];
\draw[->, line width=0.4mm] (0,-1) -- (0-0.4,-1);
\draw[fill] (0,-0) circle [radius = 0.1];
\draw[->, line width=0.4mm] (0,-0) -- (0-0.4,0);
\draw[fill] (0,+1) circle [radius = 0.1];
\draw[->, line width=0.4mm] (0,+1) -- (0,+1-0.4);
\draw[fill] (0,+2) circle [radius = 0.1];
\draw[->, line width=0.4mm] (0,+2) -- (0+0.4,+2);

\draw[fill] (1,-2) circle [radius = 0.1];
\draw[->, line width=0.4mm] (1,-2) -- (1+0.4,-2);
\draw[fill] (1,-1) circle [radius = 0.1];
\draw[->, line width=0.4mm] (1,-1) -- (1,-1-0.4);
\draw[fill] (1,-0) circle [radius = 0.1];
\draw[->, line width=0.4mm] (1,-0) -- (1,0+0.4);
\draw[fill] (1,+1) circle [radius = 0.1];
\draw[->, line width=0.4mm] (1,+1) -- (1-0.4,+1);
\draw[fill] (1,+2) circle [radius = 0.1];
\draw[->, line width=0.4mm] (1,+2) -- (1,+2+0.4);

\draw[fill] (2,-2) circle [radius = 0.1];
\draw[->, line width=0.4mm] (2,-2) -- (2,-2+0.4);
\draw[fill] (2,-1) circle [radius = 0.1];
\draw[->, line width=0.4mm] (2,-1) -- (2+0.4,-1);
\draw[fill] (2,-0) circle [radius = 0.1];
\draw[->, line width=0.4mm] (2,-0) -- (2+0.4,0);
\draw[fill] (2,+1) circle [radius = 0.1];
\draw[->, line width=0.4mm] (2,+1) -- (2,+1-0.4);
\draw[fill] (2,+2) circle [radius = 0.1];
\draw[->, line width=0.4mm] (2,+2) -- (2,+2-0.4);

\draw[fill, color = red] (0,0.1) circle [ radius = 0.1];
\draw[->, line width=0.4mm, color = red] (0,0.1) -- (0,0.1+0.4);
\end{tikzpicture} 
\caption{\label{Upper_Fig_Pairing} An example of a pair $\omega^+$ and $\omega^{-}$ from the proof of Lemma \ref{Upper_Lem_Coupling}.  }
\end{center}
\end{figure}
% End of fig_coupling

We are now going to construct a coupling. Define $\eta^+$ to be a process started at $\eta_0^+ = \omega^+$ and evolving according to the dynamics of $U$. Define $\eta^-_0 = \omega^-$ and have $\eta^-$ follow exactly the same sequence of swaps and shifts as $\eta^+$. Trivially $\eta^-$ also follows the dynamics of $U$. Then we can write
\begin{equation}
\label{Upper_Coupling_Correlation}
\Phi_U(t; \phi) = \frac{1}{2d} \int_{\Omega} \E_U[\phi(\eta^+_t) - \phi(\eta^-_t)  | \omega ] \pi(d\omega). 
\end{equation} 
Furthermore, since the dynamics of $U$ are not affected by the environment, at all times $\eta^+$ and $\eta^-$ disagree only at a single site, which we will call the defect. Furthermore $\eta^+$ and $\eta^-$ maintain an equal displacement from the defect for all times.

Now let $W_t$ denote the displacement of the walker from the defect at time $t$. When the walker is away from defect, $W_t$ evolves as rate 1 simple random walk (SRW). When $W_t = 0$, however, the walker has the defect in its hand with probability $\tfrac{1}{2}$ (this probability persists due to the initial swap). If the defect is either in the walker's hand or at the site of the walker, then under the dynamics of $U$ the defect moves with the walker with probability $\tfrac{1}{2}$. Therefore $W_t$ remains at $0$ with probability $\tfrac{1}{2}$ so is indeed a SRW with sticky origin. 

The proof is now complete by (\ref{Upper_Coupling_Correlation}) and noting that $\phi(\eta^+_t) - \phi(\eta^-_t)= 2 \cdot \mathbf{1}_{W_t = 0}$.
\end{proof}

Calculating $\widehat{\Phi}_U(\lambda, \phi)$ is straightforward, but for this we will require the Fourier transform, which we will denote
\[
\mathcal{F}u(p)
	:= \sum_{x \in \Z^d} u(x) e^{2\pi i \langle p , x \rangle },
		\qquad \textrm{for } p \in [0,1]^d.
\]
When $u$ is a function of two variables we will write 
\[
\mathcal{F}_1 u(p; y) := \mathcal{F}u(\cdot; y)(p),
	\qquad \qquad 
	\mathcal{F}_2 u(x; q) := \mathcal{F}u(x; \cdot)(q),
\]
for $p, q \in [0,1]^d$ and $x, y \in \Z^d$. The function
\begin{equation}
\label{eq:Upper_DefOfZeta}
\zeta_d(p) := 1 - \frac{1}{d} \sum_{j=1}^{d} \cos(2\pi p_j),
	\qquad \textrm{for } p \in [0,1]^d,
\end{equation}
will also be helpful.

%% LEM: Sticky RW kernel calculation
\begin{lem}
\label{Upper_Lem_LaplaceOfCorrelation}
For $\lambda > 0$, define 
\[
I_\lambda := \int_{[0,1]^d} \frac{dp}{\lambda + \zeta_d(p)}.
\]
Then as $\lambda \to 0$
\[
\widehat{\Phi}_U(\lambda; \phi) 
	= \frac{2I_\lambda}{d(1 + 2 \lambda I_\lambda)}
	= O\Big( \int^{1}_{0} \frac{r^{d-1}}{\lambda + r^2} dr \Big), 
\]
and likewise
\[\widehat{\Phi}_U(\lambda; \widetilde{\phi})  = O\Big( \int^{1}_{0} \frac{r^{d-1}}{\lambda + r^2} dr \Big).\]
Hence we have 
\[
	\widehat{E}_{G_\varepsilon}(\lambda), \widehat{E}_{\widetilde{G}_\varepsilon}(\lambda)  = \begin{cases}
O(\lambda^{-5/2}), & \textrm{if }d=1\\
O(\lambda^{-2} \log\lambda^{-1}), & \textrm{if }d=2\\
O(\lambda^{-2}), & \textrm{if }d\geq3.
\end{cases}
\]
\end{lem}

\begin{proof}
Let $k$ denote the transition kernel of the SRW with sticky origin
\[
k_t(x;y) = \P(W_t = y | W_0 = x),
\]
so $\Phi_U(t,\phi) = d^{-1} k_t(0;0)$. Suppose $x \neq 0$, then by considering the first waiting time of $W$
\[
\widehat{k}_\lambda(x;0)
	= \E_{W_0 = x} \int^\infty_0 e^{-\lambda t} \mathbf{1}_{W_t = 0} dt
	= \frac{1}{1+\lambda} \mathbf{1}_{x = 0}
		+ \frac{1}{1+\lambda} \frac{1}{2d} \sum_{x_0 \sim x} \widehat{k}_{\lambda}(x_0; 0).
\]
Taking a Fourier transform in the $x$-variable gives
\[
(\lambda + \zeta_d(p)) \mathcal{F}_1 \widehat{k}_\lambda(p;0)
	= \widehat{k}_{\lambda}(0; 0)
	 - \frac{1}{2d} \sum_{x_0 \sim 0} \widehat{k}_\lambda(x_0;0),
\]
where we have subtracted off the $x = 0$ term. By considering the waiting time at the origin we also have
\[
\frac{1}{2d} \sum_{x_0 \sim 0} \widehat{k}_{\lambda}(x_0; 0)
	= (1 + 2\lambda) \widehat{k}_{\lambda}(0; 0) - 2.
\]
and setting into the equation above gives 
\begin{equation}
\label{eq:Upper_howtocalclaplace}
\mathcal{F}_1 \widehat{k}_\lambda(p; 0) 
	= \frac{2 - 2 \lambda \widehat{k}_\lambda(0;0) }{ \lambda + \zeta_d(p) }.	
\end{equation}
By integrating over $p \in [0,1]^d$ and rearranging we obtain
\[
\widehat{k}_\lambda(0;0) = \frac{2I_\lambda}{1 + 2\lambda I_\lambda},
\]
as required.

The remainder of the result follows by an elementary calculation, Lemma \ref{Resolvent_Lem_Yaglom} and inequality (\ref{eq:Upper_CorrelationBound}). 
\end{proof}

%%%%%%%%%%%%%%%%%%%%%%%%%%%%%%%%%%%%%%%%%%%%%%%%%%%%%
%%%%%%%%%%%%%%%%%%%%%%%%%%%%%%%%%%%%%%%%%%%%%%%%%%%%%
%%%%%%%%%%%%     5. THE LOWER BOUND        %%%%%%%%%%
%%%%%%%%%%%%%%%%%%%%%%%%%%%%%%%%%%%%%%%%%%%%%%%%%%%%%
%%%%%%%%%%%%%%%%%%%%%%%%%%%%%%%%%%%%%%%%%%%%%%%%%%%%%
\section{Lower bound for $d=1$ and totally anisotropic $d=2$}
\label{Sect_LowerBound}

%% Pick specific psi, only for d=1 and d=2 anisotropic 

The derivation of the lower bound is also based on the variational formula in (\ref{eq:Resolvent_VarForm}), and we begin simply by choosing any $\psi \in L^2(\Omega, \pi)$ and dropping the supremum:
\begin{align}
\label{eq:Lower_VarForm}
\widehat{\Lambda}(\lambda; \phi_i)
	&= 2\lambda^{-2} \langle\phi_i, (\lambda - G_\varepsilon)^{-1} \phi_i\rangle_{\pi} \\
	&\geq 
	2\lambda^{-2} \Big\{ 2\langle\phi_i, \psi\rangle_\pi - \langle\psi, (\lambda - S_\varepsilon)  \psi\rangle_\pi
	- \langle A\psi, (\lambda - S_\varepsilon)^{-1} A\psi\rangle_{\pi} \Big\}.\nonumber
\end{align}
Our strategy will be to choose a family of $\psi$ (parametrised by $\lambda$) that give sufficient asymptotic growth of the lower bound.  

Since we will eventually be using the work in this section to prove the lower bounds in  Theorems \ref{Intro_Thm_D1} and \ref{Intro_Thm_D2}, it will suffice to take the initial measure, $\pi_p$, to be such that $p_1 = 1$. In $d=1$ this is the only drift-free measure, and in $d=2$ this is no loss of generality in the totally anisotropic case, by symmetry. Throughout the following we only consider $d = 1$ or $2$. 

%% REM: Iso harder
\begin{rem}[The $d=2$ cases]
\label{Lower_Rem_Warning}
In two-dimensions, the totally anisotropic case is substantially easier to analyse than the non-totally anisotropic case, which is the subject of the next section. This is reflected in the relative strength of the orders of the bounds in Theorem \ref{Intro_Thm_D2}. The key observation is that for the totally anisotropic initial condition it suffices to prove weak bounds on the main correlation function (Definition \ref{Lower_Def_4dDefectWalk} and Lemma \ref{Lower_Lem_Correlation}) in this section. Indeed the relevant bound (Lemma \ref{Lower_Lem_AntiSymm}) follows from a simple application of Cauchy--Schwarz. This method fails when we have an initial condition that is not totally anisotropic. 
\end{rem}

In order to perform exact computations, we will choose test functions $\psi$ of the form
\[
\psi(\omega) = \sum_{x \in \mathbb{Z}^d}  u(x) \omega(x)_1 + u(\star) \omega(\star)_1,
\]
where $u \in \ell^2(\mathbb{Z}^d_\star; \mathbb{R})$, which guarantees $\psi \in L^2(\Omega, \pi)$. We will further restrict $u$ so that
\begin{equation}
\label{eq:Lower_Restrictions}
u(x) = u(-x) \textrm{ for all }x,
	 \textrm{ and }
	u(0) = u(\pm e_1) = u(\star).
\end{equation}
We insist that $u$ be an even function to make the Fourier transform $\mathcal{F}u$ real-valued. Throughout let $\Vert \cdot \Vert_2$ denote the norm on $\ell^2(\Z^d_\star; \R)$ and define the gradients $\nabla^+$ and $\nabla^-$ by
\[
\nabla^\pm_k u(x) := \begin{cases}
u(x \pm e_{k})-u(x) & \textrm{if }x\neq\star,\\
0 & \textrm{if }x=\star,
\end{cases}
\]
and $\nabla_k u (x) := \nabla^+_k u(x) - \nabla^-_k u(x)$.

Notice that from Lemma \ref{Resolvent_Lem_Yaglom} it suffices to prove lower bounds on just the first component $\widehat{\Lambda}(\lambda; \phi_1)$. Therefore throughout this section we will write $\phi(\omega)$ and $\widetilde{\phi}(\omega)$ to refer to the first components $\phi(\omega)_1$ and $\widetilde{\phi}(\omega)_1$, which we recall are (in this short-hand)
\[
\mathrm{M1}_\varepsilon \ : \ \phi(\omega) = \omega(\star)_1,
\qquad \textrm{and} \qquad
\mathrm{M2}_\varepsilon \ : \ \widetilde{\phi}(\omega) = \tfrac{1}{2}(\omega(0)_1 + \omega(\star)_1). 
\]

%% LEM: Some computations
\begin{lem}[Some computations]
\label{Lower_SomeComputations}
With the notation above, we have:
\begin{enumerate}[(i)]
\item $\langle\phi, \psi\rangle_\pi = u(0) = \langle\widetilde{\phi}, \psi\rangle_\pi$,
\item $\langle\psi, \psi\rangle_\pi =   \Vert u \Vert_{2}^2$,  
\item $\langle\psi, S_\varepsilon \psi\rangle_\pi = - \tfrac{1+\varepsilon}{2} \Vert \nabla_1^+ u \Vert_2^2 - \tfrac{\varepsilon}{2} \Vert \nabla_2^+ u \Vert_2^2= \langle\psi, \widetilde{S}_\varepsilon \psi\rangle_\pi$,
\item $A\psi(\omega) = -\tfrac{1}{4} \sum_{x \in \mathbb{Z}^d} \nabla_1 u(x) \{ \omega(0)_1 + \omega(\star)_1 \} \omega(x)_1
	= \widetilde{A} \psi (\omega).$
\end{enumerate}
\end{lem}

\begin{proof}
(i), (ii) Trivial given the product form of $\pi$, and that $u(\star) = u(0)$.

(iii) Notice that $\psi(\sigma \omega) = \psi(\omega)$, for all $\omega \in \Omega$, and
\begin{align*}
\psi(\tau_{\star} \omega)
	 = \psi(\sigma \tau_{\star} \omega)
	&=  \sum_{x \in \Z^d} u(x - \omega(\star)) \omega(x)_1 + u(0)\omega(\star)_1 \\
\psi(\tau^{-1}_{\star}\omega)
	= \psi(\sigma \tau^{-1}_{\star} \omega)
	&= \sum_{x \in \Z^d} u(x + \omega(\star)) \omega(x)_1 + u(0)\omega(\star)_1 \\
\psi(\tau_{\star}\sigma\omega)
	= \psi(\sigma\tau_{\star}\sigma\omega)
	&= \sum_{x \neq 0} u(x - \omega(0)) \omega(x)_1 + u(0)\omega(0)_1 + u(-\omega(0))\omega(\star)_1 \\
\psi(\tau^{-1}_{\star}\sigma\omega)
	= \psi(\sigma\tau^{-1}_{\star}\sigma\omega)
	&=  \sum_{x \neq 0} u(x + \omega(0)) \omega(x)_1 + u(0)\omega(0)_1 + u(\omega(0))\omega(\star)_1.
\end{align*}
Integrating any of these terms against $\psi(\omega)\pi(d\omega)$ gives
\begin{align*}
\sum_{x \in \mathbb{Z}^d} u(x + e_1) u(x) + u(0)^2
	&= \Vert u \Vert_2^2	- \tfrac{1}{2} \Vert \nabla_1^+ u \Vert_2^2,
\end{align*}
and after subtracting $(\psi, \psi)_\pi$ we obtain $-\tfrac{1}{2} \Vert \nabla_1^+ u \Vert_2^2$. Repeating the analysis for $U$ gives 
\[
\langle\psi, U \psi\rangle_\pi
	= -\tfrac{1}{2} \Vert \nabla_1^+ u \Vert_2^2 -\tfrac{1}{2} \Vert \nabla_2^+ u \Vert_2^2,
\]
so we have the result. 

(iv). The result for $\widetilde{A} \psi$ follows from taking appropriate differences of the eight terms above and using the identity
\[
u(x - z) - u(x + z) = -z_1 \nabla_1 u(x),
\]
for $z = \pm e_1$. For $A\psi$, from the calculations in (iii) we have
\begin{multline*}
A\psi(\omega)
	= \frac{1}{2}\sum_{x \in \mathbb{Z}^d} \{ u(x - \omega(\star)) - u(x + \omega(0)) \} \omega(x)_1
		\\+ \frac{1}{2}\{ u(0) - u(\omega(0)) \} \omega(\star)_1 
		- \frac{1}{2} \{ u(0) - u(\omega(0)) \} \omega(0)_1.
\end{multline*}
The final two terms vanish due to the restrictions in (\ref{eq:Lower_Restrictions}), then we are done by noting that
\[
u(x - \omega(\star)) - u(x + \omega(0))
	= -\nabla_1 u(x) \frac{\omega(0)_1 + \omega(\star)_1}{2}.
\]
\end{proof}

As with the upper bound, we will operate in the Fourier domain. Since $u$ is symmetric
\[
\mathcal{F}u(p) = \sum_{x \in \Z^d} u(x) \cos(2\pi \langle x , p \rangle) \in \R, 
\]
for all $p \in [0,1]^d$. Easy calculations give
\begin{gather*}
\label{eq:Lower_HandyCalculation_I}
 \mathcal{F}[ \nabla_1^+ u](p) 
	= ( e^{-2\pi i p_1} - 1) \mathcal{F}u(p),
	\qquad 
	\mathcal{F}[ \nabla_1 u ](p) 
		= -2i\sin(2\pi p_1) \mathcal{F}u(p). 
\end{gather*}
With these we can now express the value of symmetric part in (\ref{eq:Lower_VarForm}) as follows.

%% LEM: Symm part
\begin{lem}[Symmetric part]
\label{Lower_Lem_SymmPart}
Let $\zeta_d$ be as in (\ref{eq:Upper_DefOfZeta}). Then we have
\[
2\langle\phi, \psi\rangle_\pi - \langle\psi, (\lambda - S_\varepsilon)  \psi\rangle_\pi 
	\geq \int_{[0,1]^d} \Big( 2\mathcal{F}u(p) 
	- \{2 \lambda + (1+\varepsilon) d \zeta_d(p)\} \mathcal{F}u(p)^2
\Big)dp,
\]
and this lower bound holds for $\widetilde{\phi}$ and $\widetilde{S}_\varepsilon$ also.
\end{lem}

\begin{proof}
Beginning from Lemma \ref{Lower_SomeComputations}, notice that 
\[
u(0) = \int_{[0,1]^d} \mathcal{F}u(p)dp \qquad
	\textrm{and} \qquad
	\Vert u \Vert_{\ell^2(\mathbb{Z}^d_\star)}^2
		\leq 2 \Vert u \Vert_{\ell^2(\mathbb{Z}^d)}^2
		= 2 \Vert \mathcal{F}u \Vert^2_{L^2([0,1]\times[0,1])},
\]
where we have used $u(\star) = u(0)$ in  obtaining the upper bound. From (\ref{eq:Lower_HandyCalculation_I}) we have
\begin{align*}
\Vert \nabla^+_j u \Vert_{\ell^2}^2
	= \Vert \mathcal{F}\nabla^+_j u  \Vert_{L^2}^2
	&= \int_{[0,1]^d} | e^{-2\pi i p_j} - 1|^2 |\mathcal{F} u(p)|^2 dp \\
	&= 2\int_{[0,1]^d} (1 - \cos(2\pi p_j))\mathcal{F} u(p)^2 dp,
\end{align*}
where we have used that $\mathcal{F} u(p) \in \mathbb{R}$. Combining the results with Lemma \ref{Lower_SomeComputations} completes the proof. 
\end{proof}

We now proceed with the more challenging anti-symmetric part in (\ref{eq:Lower_VarForm}), but note from Lemma \ref{Lower_SomeComputations} that we only need to refer to $A\psi$. Begin using $S_\varepsilon \leq \varepsilon U$ as in the derivation of (\ref{eq:Upper_CorrelationBound}) to arrive at the upper bounds
\begin{align}
\label{Lower_FormForApsi}
\langle A\psi, (\lambda - S_\e)^{-1} A\psi\rangle_\pi
	&\leq \varepsilon^{-1} \widehat{\Phi}_U(\varepsilon^{-1} \lambda; A \psi) \\
\langle \tilde{A}\psi, (\lambda - \widetilde{S}_\e)^{-1} \tilde{A}\psi \rangle_\pi = \langle A\psi, (\lambda - \widetilde{S}_\e)^{-1} A\psi\rangle_\pi & \leq \varepsilon^{-1} \widehat{\Phi}_U(\varepsilon^{-1} \lambda; A \psi).
\end{align}
Using the explicit form in Lemma \ref{Lower_SomeComputations} (iv) gives
\begin{equation}
\label{eq:Lower_tobecauchy}
\Phi_U(t; A\psi) 
	= \mathbb{E}_{U} [A\psi(\eta_0) A\psi(\eta_t) ] 
	= \frac{1}{16}\sum_{x,y \in \mathbb{Z}^d} \nabla_1 u(x) \nabla_1 u(y) C(t; x, y),
\end{equation}
where $C$ is the order-four correlation function
\begin{equation}
\label{eq:Lower_4point_form_}
C(t; x, y)
	= \mathbb{E}_U [ \{\eta_0(0)_1 + \eta_0(\star)_1 \} \eta_0(x) \{\eta_t(0)_1 + \eta_t(\star)_1 \} \eta_t(y)  ]. 
\end{equation}

Again a coupling argument allows us to write $C$ as the transition kernel of a defect process. This time, however, the underlying random walk is $2d$-dimensional and more complicated.

%% DEF: Defect walk graph
\begin{defn}[Defect walk]
\label{Lower_Def_4dDefectWalk}
Define a graph $\mathcal{G} = (\mathsf{V}, \mathsf{E})$ with vertex set
\[
\mathsf{V} := (\Z^{2d}  \smallsetminus \{(x,x) : x \in \Z^d \} ) \cup \{ 0 \}
\]
and edge set, $\mathsf{E} \subseteq \{ \{ v,w \} : v,w \in \mathsf{V} \}$, given by
\begin{align*}
\mathsf{E} 
	&:= \{ \{(x, y), (x + e, y + e)\} : e \in \mathcal{E}, x,y \in \Z^d, (x, y) \neq (0,0) \} \\
		&\qquad \cup \{ \{(0, 0), (0, e)\}: e \in \mathcal{E} \}
		 \cup \{ \{(0, 0), (e, 0)\}: e \in \mathcal{E} \}\\
		 &\qquad\cup \{ \{(x, 0), (x +e, 0)\}: e \in \mathcal{E}, x \in \Z^d \}
		 \cup \{ \{(0, y), (0, y +e)\}: e \in \mathcal{E}, y \in \Z^d \}, 
\end{align*}
where we recall that $\mathcal{E}$ is the set of canonical unit vectors (see Figure \ref{Fig_Lower_Defect_Graph}). Define $D$ to be a rate 1 nearest neighbour symmetric random walk on $\mathcal{G}$ and denote its transition kernel by
\[
K_t(u; v)
	:= \P(D_t = v | D_0 = u),
		\qquad \textrm{for } u,v \in \mathsf{V}. 
\]
\end{defn}

%% FIG: Defect graph
%\input{fig_defect_graph}
\begin{figure}
\begin{center}
\begin{tikzpicture}[scale = 0.65]
\draw[<->](0,3.5) -- (0,0) -- (3.5,0);
\draw[-](0,-3.5) -- (0,0) -- (-3.5,0);
\foreach \i in {-3,...,3} {
	\foreach \j in {-3, ..., 3} {	
		\draw[fill] (\j,\i) circle [radius = 0.1];
	}
}
\foreach \i in {-3,...,-1} {
	\draw[fill, white] (\i, \i) circle [radius = 0.15];
	\draw[fill, white] (-\i, -\i) circle [radius = 0.15];
}
\draw[-](-2.8,3.2) -- (-3.2,2.8);
\draw[-](-1.8,3.2) -- (-3.2,1.8);
\draw[-](-0.8,3.2) -- (-3.2,0.8);
\draw[-](0.2,3.2) -- (-3.2,-0.2);
\draw[-](1.2,3.2) -- (-3.2,-1.2);
\draw[-](2.2,3.2) -- (-3.2,-2.2);
\draw[-](3.2,2.2) -- (-2.2,-3.2);
\draw[-](3.2,1.2) -- (-1.2,-3.2);
\draw[-](3.2,0.2) -- (-0.2,-3.2);
\draw[-](3.2,-0.8) -- (0.8,-3.2);
\draw[-](3.2,-1.8) -- (1.8,-3.2);
\draw[-](3.2,-2.8) -- (2.8,-3.2);
\end{tikzpicture}
\caption{\label{Fig_Lower_Defect_Graph} The graph $\mathcal{G}$ from Definition \ref{Lower_Def_4dDefectWalk} in the case $d = 1$. }
\end{center}
\end{figure}
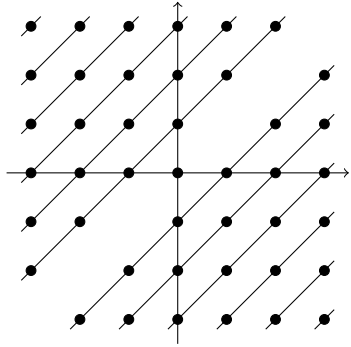
% End of fig_defect_graph

%% LEM: 4-point correlation
\begin{lem}[Correlation]
\label{Lower_Lem_Correlation}
For all $t > 0$ and $x \neq 0 \neq y$ 
\[
C(t; x, y)
	=\frac{1}{4} \Big( K_t(x,0; y,0) 
		+  K_t(x,0; 0,y) \Big), 
\]
where $K$ is the transition kernel from Definition \ref{Lower_Def_4dDefectWalk}. 
\end{lem}

\begin{proof}
We follow the same proof as in Lemma \ref{Upper_Lem_Coupling}, except that we introduce two defects instead of one. By symmetry of the first swap we have
\[
C(t; x, y)
	= \frac{1}{2} \int_{\Omega} \{ \omega(0)_1 + \omega(\star)_1 \} \omega(x)_1 \rho(\omega) \pi(d\omega)
	= \int_{\Omega}  \omega(0)_1 \omega(x)_1 \rho(\omega) \pi(d\omega),
\]
where 
\[
\rho(\omega)
	:= \mathbb{E}_U[\{ \eta_t(0) + \eta_t(\star) \} \eta_t(y) | \eta_0 = \omega ].
\]
Introduce the notation
\[
\omega^{s_{1}s_{2}}(z):=\begin{cases}
s_{1}e_{j} & \textrm{if }z=x,\\
s_{2}e_{k} & \textrm{if }z=0,\\
\omega(z) & \textrm{otherwise},
\end{cases}\qquad \qquad \textrm{for } s_1, s_2 \in \{-,+\} .
\]
By conditioning on the values of $\omega(x)$ and $\omega(0)$ (recall that $x \neq 0$) we have
\begin{equation}
\label{eq:Lower_Correlation_ONLY}
C(t; x, y)
	=\frac{1}{16} \int_{\Omega} \Big( 
		\rho(\omega^{++}) + \rho(\omega^{--}) 
			- \rho(\omega^{+-})- \rho(\omega^{-+})
	\Big) \pi(d\omega).
\end{equation}

We will now construct a coupling. For $s_1, s_2 \in \{-1, +1\}$, let $\eta^{s_1, s_2}_0 = \omega^{s_1, s_2}$ and have all the $\eta^{s_1, s_2}$ evolving according to the same sequence of swaps and steps generated by a realisation of the dynamics under $U$. Now any pair of these processes have two defects, so let $D^1_t$ denote the displacement of the defect that is initially at $x$ and $D^2_t$ the displacement of the defect initially at $0$. Then $D_t := (D^1_t, D^2_t)$ has the dynamics described in Definition \ref{Lower_Def_4dDefectWalk}.

For short-hand also introduce the notation
\begin{align*}
\alpha(\omega) &:= \omega(y)_1 \\
\beta(\omega) &:= \omega(0)_1 + \omega(\star)_1 \\
\Gamma(\omega) &:= \alpha(\omega) \beta(\omega).
\end{align*}
By considering the positions of the defect, $D$, we have
\begin{align}
\label{eq:Lower_Coupling1}
\Gamma(\eta^{++}_t)
	&= \{ \alpha(\eta^{+-}_t) +  2 \cdot \mathbf{1}_{D^2_t = y} \}\nonumber
		\{ \beta(\eta^{+-}_t) + 2 \cdot  \mathbf{1}_{D^2_t = 0}  \} \\
	&= \Gamma(\eta^{+-}_t)
		+ 2 \alpha(\eta^{+-}_t) \mathbf{1}_{D^2_t = 0}
		+ 2 \beta(\eta^{+-}_t) \mathbf{1}_{D^2_t = y}
\end{align}
where we have used $y \neq 0$, and so by symmetry
 \begin{align}
 \label{eq:Lower_Coupling2}
\Gamma(\eta^{-+}_t)
	&= \Gamma(\eta^{--}_t)
		+ 2 \alpha(\eta^{--}_t) \mathbf{1}_{D^2_t = 0}
		+ 2 \beta(\eta^{--}_t) \mathbf{1}_{D^2_t = y}.
\end{align}
By using the fact
\[
\alpha(\eta^{+-}_t) - \alpha(\eta^{--}_t)
	= 2\cdot \mathbf{1}_{D^1_t = y},
		\qquad
\beta(\eta^{+-}_t) - \beta(\eta^{--}_t)
 	= 2\cdot \mathbf{1}_{D^1_t = 0},
\]
subtracting (\ref{eq:Lower_Coupling2}) from  (\ref{eq:Lower_Coupling1}) gives
\[
\Gamma(\eta^{++}_t) + \Gamma(\eta^{--}_t) - \Gamma(\eta^{-+}_t)
	- \Gamma(\eta^{+-}_t) 
	= 
	4 \cdot \mathbf{1}_{D_t = (y,0)}
	+ 4 \cdot \mathbf{1}_{D_t = (0,y)}.
\]
Taking a conditional expectation with respect to $\omega$ and setting into (\ref{eq:Lower_Correlation_ONLY}) completes the result. 
\end{proof}

Evaluating the Laplace transform of the kernel in Definition \ref{Lower_Def_4dDefectWalk} could be tricky. Fortunately a simple application of Cauchy--Schwarz allows us to obtain a bound on (\ref{eq:Lower_tobecauchy}) in terms of the transition of a $d$-dimensional simple random walk with sticky origin. The key observation is that if we project the walk $D$ from Definition \ref{Lower_Def_4dDefectWalk} down to one of its coordinates, then the corresponding process is a simple random walk with sticky origin. This is easiest to see from Figure \ref{Fig_Lower_Defect_Graph}.

%% LEM: 4-point easy estimate
\begin{lem}[Anti-symmetric part]
\label{Lower_Lem_AntiSymm}
With $I_\lambda$ as defined in Lemma \ref{Upper_Lem_LaplaceOfCorrelation}, we have
\[
|\widehat{\Phi}_U(\lambda; A\psi)| 
	\leq \int_{[0,1]^d} \tfrac{1}{4}I_\lambda \sin^2(2\pi p_1 ) \mathcal{F} u(p)^2 dp.
\]
\end{lem}

\begin{proof}
Begin by applying Cauchy--Schwarz to (\ref{eq:Lower_tobecauchy}):
\begin{equation}
\label{eq:Lower_BoundCorrelation_I}
|\Phi_U(t, A\psi)|
	\leq \sum_{x,y\in \mathbb{Z}^d} \nabla_1 u(x)^2 C(t;x,y)
	= \sum_{x \in \mathbb{Z}^d} \nabla_1 u(x)^2 \sum_{y \in\mathbb{Z}^d} C(t; x, y),
\end{equation}
where we have used the symmetry in the $x$ and $y$ variables and the positivity of $C$ from Lemma \ref{Lower_Lem_Correlation}. Writing $k_t$ for the transition kernel of a one-dimensional random walk with sticky origin (as in Lemma \ref{Upper_Lem_Coupling} and \ref{Upper_Lem_LaplaceOfCorrelation}), by the remark made above we have
\begin{align}
\label{eq:Lower_BoundCorrelation_II}
\sum_{y \in \mathbb{Z}^d} C(t; x, y)\nonumber
	&= \frac{1}{4} \sum_{y \in \mathbb{Z}^d} K_t(x,0;y,0) + K_t(x,0;0,y) \\
	&= \tfrac{1}{4} \mathbb{P}(D_t^2 = 0 | D_0 = (x,0))\nonumber
		+ \tfrac{1}{4}\mathbb{P}(D_t^1 = 0 | D_0 = (x,0))\\
	&= \tfrac{1}{4} k_t(0; 0) + \tfrac{1}{4} k_t(x; 0). 
\end{align}

Returning to (\ref{eq:Upper_howtocalclaplace}), by multiplying through by $e^{-2\pi i \langle p, x\rangle}$ and integrating (i.e.~inverting the Fourier transform) and using the positivity of $k_t$ we have 
\[
\widehat{k}_\lambda(x; 0)
	= 2(1 - \lambda \widehat{k}_\lambda(0;0))\int_{[0,1]^d} \frac{e^{-2\pi i \langle p, x\rangle}}{\lambda  + \zeta_d(p)}dp
	\leq 2 \int_{[0,1]^d} \frac{dp}{\lambda  + \zeta_d(p)}
	= 2I_\lambda.
\]
Therefore setting into (\ref{eq:Lower_BoundCorrelation_I}) using (\ref{eq:Lower_BoundCorrelation_II}) gives
\[
|\widehat{\Phi}_U(\lambda, A\psi)|
	\leq \tfrac{1}{16} I_\lambda \Vert \nabla_1 u \Vert_2^2.
\]
The result follows by Plancherel's theorem and $\mathcal{F}[\nabla_1 u](p)   = -2i \sin(2\pi p_1) \mathcal{F} u (p)$.
\end{proof}

The final step in this section is to combine Lemmas \ref{Lower_Lem_SymmPart} and \ref{Lower_Lem_AntiSymm} and to choose $u$ appropriately to get a non-trivial lower bound. This is done through a simple point-wise optimisation of a quadratic integrand. Additional work must be done to ensure that the choice of $u$ satisfies the constraints in (\ref{eq:Lower_Restrictions}).

%% LEM: The main lower bound
\begin{lem}[The main lower bound]
\label{Lower_Lem_TheMainLowerBound}
There exists a constant $c_0 > 0$ such that
\[
\widehat{E}_{G_\varepsilon}(\lambda) ,\ \widehat{E}_{\widetilde{G}_\varepsilon}(\lambda)
	\geq c_0 \lambda^{-2} \int_{[0,1]^d} \frac{dp}{\lambda + |p|^2 +  I_\lambda p_1^2  },
\]
for all $\lambda > 0$ sufficiently small. 
\end{lem}

\begin{proof}
Using (\ref{eq:Lower_VarForm}), Lemma \ref{Lower_Lem_SymmPart}, (\ref{Lower_FormForApsi}) and Lemma \ref{Lower_Lem_AntiSymm} together gives 
\[
\widehat{E}_{G_\varepsilon}(\lambda)
	\geq 2\lambda^{-2} \int_{[0,1]^d} (
		2\mathcal{F}u(p) - H_\lambda(p) \mathcal{F}u(p)^2) dp
\]
where 
\[
H_\lambda(p) 
	:= 2 \lambda + (1+\varepsilon) d\zeta_d(p) + \tfrac{1}{4} \varepsilon^{-1} I_{\varepsilon^{-1} \lambda} \sin^2(2\pi p_1).
\]
The integrand is maximised at $u = v$ where $v$ is such that  $\mathcal{F}v(p) = H_\lambda(p)^{-1}$. Since $H_\lambda$ is even and bounded, there is guaranteed to exist such a $v$ that is real-valued and even, whereby the integrand takes value $H_\lambda(p)^{-1}$. 

This choice of $v$, however, does not satisfy (\ref{eq:Lower_Restrictions}), and so we define 
\[
u(x) := v(x) + \theta \delta(x),
	\qquad \theta := v(e_1) - v(0),
\]
where $\delta(x) = \mathbf{1}_{x=0}$. Then $u$ satisfies (\ref{eq:Lower_Restrictions}) and 
$\mathcal{F} u (p) = \mathcal{F} v (p) + \theta$, so setting into the integral gives
\[
\widehat{E}_{G_\varepsilon}(\lambda)
	\geq 2\lambda^{-2} \int_{[0,1]^d} (H_\lambda(p)^{-1} - \theta^2 H_\lambda(p) )dp.
\]
Notice, however, that $H_\lambda(p) = O(I_\lambda)$ and $H_\lambda(p) = \Omega(I_\lambda p_1^2)$ for all $p \in [0,1]^d$, hence
\[
|\theta| 
	= \int_{[0,1]^d} \frac{|1-\cos(2\pi p_1)|}{H_\lambda(p)} dp
	= O\Big( \int_{[0,1]^d} \frac{p_1^2}{I_\lambda p_1^2} dp \Big)
	= O(I_\lambda^{-1}).
\]
With this we conclude that $\theta^2 H_\lambda(p) = O(I_\lambda^{-1}) = o(1)$, by Lemma \ref{Upper_Lem_LaplaceOfCorrelation} for $d = 1$ and $2$, therefore
\[
\widehat{E}_{G_\varepsilon}(\lambda) 
	\geq 2\lambda^{-2} \int_{[0,1]^d} \frac{dp}{H_\lambda(p)} + o(\lambda^{-2}).
\]
We are then done by using $\zeta_d(p) = O(|p|^2)$ and $|\sin(2\pi p_1)| = O(p_1)$.
\end{proof}

%%%%%%%%%%%%%%%%%%%%%%%%%%%%%%%%%%%%%%%%%%%%%%%%%%%%%
%%%%%%%%%%%%%%%%%%%%%%%%%%%%%%%%%%%%%%%%%%%%%%%%%%%%%
%%%%%%%%%%%%     6. COMPLETING PROOFS      %%%%%%%%%%
%%%%%%%%%%%%%%%%%%%%%%%%%%%%%%%%%%%%%%%%%%%%%%%%%%%%%
%%%%%%%%%%%%%%%%%%%%%%%%%%%%%%%%%%%%%%%%%%%%%%%%%%%%%
\section{Completing the proof of Theorem \ref{Intro_Thm_D1}, \ref{Intro_Thm_D3} and \ref{Intro_Thm_D2}}
\label{Sect_Completing}

%% wrapping up all the proofs of the main results
 
In this section we use the results derived so far to complete the proofs of the main theorems.

\subsection*{Upper bounds in all cases}

All upper bounds follow from Lemma \ref{Upper_Lem_LaplaceOfCorrelation}. \qed

\subsection*{Lower bound $d \geq 3$}

The lower bound in the time domain for Theorem \ref{Intro_Thm_D3} follows immediately from Lemma \ref{Resolvent_Lem_Yaglom} and implies the Laplace lower bound. \qed

\subsection*{Lower bound $d = 1$}

From Lemma \ref{Upper_Lem_LaplaceOfCorrelation} we know that $I_\lambda = O(\lambda^{-1/2})$ when $d=1$. Therefore the integral in Lemma \ref{Lower_Lem_TheMainLowerBound} has order at greater than
\[
\int_0^1 \frac{dp}{\lambda + \lambda^{-1/2}p^2}
	= \Omega(\lambda^{-1/4}).
\]
Hence $\widehat{E}(\lambda) = \Omega(\lambda^{-9/4})$ as required. \qed

\subsection*{Lower bound $d = 2$, totally anisotropic}

Writing the integral in Lemma \ref{Lower_Lem_TheMainLowerBound} in polar coordinates gives
\begin{align*}
\widehat{E}(\lambda) 
	&= \Omega\Big( \int^1_0 \int^{2\pi}_0 \frac{r }{\lambda + r^2 + I_\lambda r^2 \cos^2\theta} d\theta dr \Big) \\
	&= \Omega\Big( \int^1_0 \frac{r }{\sqrt{(\lambda + r^2)(\lambda + r^2 + I_\lambda r^2)}} dr \Big)\\
	&= \Omega\Big(\int^1_0\frac{r }{\sqrt{\lambda + r^2 + I_\lambda r^2}} dr \Big)
	= \Omega(\sqrt{I_\lambda}),
\end{align*}
and from Lemma \ref{Upper_Lem_LaplaceOfCorrelation} we know $I_\lambda \asymp \log(\lambda^{-1})$, which completes the proof. \qed

\subsection*{Obstruction for the two-dimensional non-totally anisotropic case}

We might think that the work in Section \ref{Sect_LowerBound} should immediately imply a similar superdiffusive lower bound for the case when $d = 2$ and $p \neq (1,0)$ or $(0,1)$. If we attempt this, however, because the environment $\omega \in \Omega$ is not constrained to one-dimensional components, the integral we reach in Lemma \ref{Lower_Lem_TheMainLowerBound} is 
\[
\lambda^{-2} \int_{[0,1]^2} \frac{dp}{\lambda + |p|^2 + I_\lambda |p|^2}
	\asymp \lambda^{-2} \int^1_0 \frac{rdr}{\lambda + \log(\lambda^{-1}) r^2}
	\asymp \lambda^{-2},
\]
which is not a superdiffusive lower bound. 

It seems necessary to follow the computations in \cite{toth_valko_2012} and to choose a linear functional $\psi$ in Section \ref{Sect_LowerBound} that depends on both components $\{ \omega(x)_1 \}_{x \in \mathbb{Z}^d_\star}$ and $\{ \omega(x)_2 \}_{x \in \mathbb{Z}^d_\star}$ and to avoid the crude Cauchy--Schwarz estimates in Lemma \ref{Lower_Lem_AntiSymm}. The difficulty here is that this necessitates a more detailed analysis of the transition kernel from Definition \ref{Lower_Def_4dDefectWalk}, which is currently out of our reach.

%%%%%%%%%%%%%%%%%%%%%%%%%%%%%%%%%%%%%%%%%%%%%%%%%%%%%
%%%%%%%%%%%%%%%%%%%%%%%%%%%%%%%%%%%%%%%%%%%%%%%%%%%%%
%%%%%%%%%%%%           REFERENCES          %%%%%%%%%%
%%%%%%%%%%%%%%%%%%%%%%%%%%%%%%%%%%%%%%%%%%%%%%%%%%%%%
%%%%%%%%%%%%%%%%%%%%%%%%%%%%%%%%%%%%%%%%%%%%%%%%%%%%%

\bibliographystyle{alpha}

\end{document}